\documentclass[12pt,reqno]{amsart}
\usepackage{amsmath}
\usepackage{amsfonts}
\usepackage{amssymb}
\usepackage{amsthm}
\usepackage{mathdots}
\usepackage{stmaryrd}
\usepackage{mathrsfs}
\usepackage{mathtools}
\usepackage{musicography}
\usepackage[makeroom]{cancel}
\usepackage{array}
\usepackage{bbm}
\usepackage[shortlabels]{enumitem}
\usepackage{hyperref}
\usepackage[noabbrev,capitalise]{cleveref}
\usepackage[margin=1.0in]{geometry}
\usepackage[most]{tcolorbox}
\usepackage{xcolor}
\usepackage{color,soul}
\usepackage{tikz}
\usetikzlibrary{
arrows.meta,shapes,automata,petri,positioning,calc,
  graphs,
  graphs.standard
}
\tikzstyle{V}=[draw, fill =black, circle, inner sep=0pt, minimum size=1.5pt]
\tikzstyle{bV}=[draw, fill =black, circle, inner sep=0pt, minimum size=4pt]
\tikzstyle{cV}=[draw, fill =white, circle, inner sep=0pt, minimum size=4pt]
\tikzstyle{over}=[draw=white,double=black,line width=2pt, double distance=.75pt]
\tikzstyle{diagram}=[line width=.75pt, scale=\SCALE]
\usepackage{macros}
\crefname{prop}{Proposition}{Propositions}
\crefname{thm}{Theorem}{Theorems}

\begin{document}
\title[Representation theory of the twisted Yangians in complex rank]{Representation theory of the twisted Yangians in complex rank}

\date{\today}

\author[A. S. Kannan]{Arun S. Kannan}
\address{Arun S. Kannan \ Chicago, IL 60606}
\email{arun\_kannan@alum.mit.edu}
\author[Shihan Kanungo]{Shihan Kanungo}
\address{San Jos\'e State University \ San Jos\'e, CA 95192}
\email{shihan.kanungo@sjsu.edu}
	
\begin{abstract} 
In 2016, Etingof defined the notion of a Yangian in a symmetric tensor category and posed the problem to study them in the context of Deligne categories. This problem was studied by Kalinov in 2020 for the Yangian $Y(\gl_t)$ of the general linear Lie algebra $\gl_t$ in complex rank using the techniques of ultraproducts. In particular, Kalinov classified the simple finite-length modules over $Y(\gl_t)$. In this paper, we define the notion of a \textit{twisted Yangian} in Deligne's categories, and we extend these techniques to classify finite-length simple modules over the twisted Yangians $Y(\o_t)$ and $Y(\xsp_t)$ of the orthogonal and symplectic Lie algebras $\o_t,\xsp_t$ in complex rank.
\end{abstract}

\maketitle
\setcounter{tocdepth}{1}
\tableofcontents
\section{Introduction}
A \textit{symmetric tensor category} (STC) is an abstraction of the representation category of a group, or more generally, an affine supergroup scheme. However, not all symmetric tensor categories arise as such representation categories; in other words, some do not \textit{fiber} over the category of complex super vector spaces $\sVec_{\mathbb{C}}$. Working with these exotic categories yields new forms of commutative algebra, algebraic geometry, and representation theory beyond the realm of (super) vector spaces. Unless stated otherwise, all symmetric tensor categories considered in this paper are defined over $\mathbb{C}$.

The most prominent examples of STCs that do not fiber over $\sVec_{\mathbb{C}}$ are the \textit{Deligne categories}, introduced in \cite{dmos_hodge_82, deligne_categories_tannakiennes_90, deligne_categories_tensorielles_02, deligne_rep_st_07}. These categories generalize classical representation categories such as $\Rep(GL_n)$, $\Rep(O_n)$, and $\Rep(Sp_n)$ to the setting where the rank parameter $n$ is replaced by a complex parameter $t \in \mathbb{C}$. An analogous construction exists for symmetric groups. In recent years, Deligne categories have been the subject of extensive study from multiple perspectives.

The present work concerns Yangians in symmetric tensor categories. The notion was introduced by Etingof in \cite{etingof_complex_rank_2_16}, who showed that to any quadratic Lie algebra $\g$ in a symmetric tensor category $\mathcal{C}$, one can associate an algebra $Y(\g)$, called its \textit{Yangian}. Although this is a categorical rephrasing of Drinfeld's original definition \cite{drinfeld_yangians_85}, Etingof provided a streamlined construction via $R$-matrices in the case of $Y(\mathfrak{gl}_t)$, where $\mathfrak{gl}_t=V \otimes V^*$ and $V$ is the tautological object in $\Rep(GL_t)$. Etingof further posed the problem of classifying the irreducible representations of $Y(\mathfrak{gl}_t)$ that have finite-length when viewed as objects in $\Rep(GL_t)$. This question was solved by Kalinov \cite{kalinov_yangians_20}, who used ultrafilter and ultraproduct techniques to bypass the lack of a triangular decomposition, following earlier applications of these methods by Harman \cite{harman_ultrafilter_16} and Deligne \cite{deligne_rep_st_07} to Deligne categories. Kalinov showed that simple finite-length representations of $Y(\mathfrak{gl}_t)$ are parametrized by pairs of sequences of Drinfeld polynomials subject to a certain stabilization condition. Independently, Utiralova \cite{utiralova_centralizer_19} proved a version of Olshanski's centralizer construction in complex rank for $Y(\mathfrak{gl}_t)$.

In this paper, we introduce the \textit{twisted Yangians} $Y(\mathfrak{sp}_t)$ and $Y(\mathfrak{o}_t)$, corresponding to the categories $\Rep(Sp_t)$ and $\Rep(O_t)$, respectively. Extending Kalinov’s methods, we classify their irreducible representations. As expected, irreducible modules are parametrized by Drinfeld polynomials, analogous to those appearing in the integer-rank theory of $Y(\mathfrak{sp}_{2n})$ and $Y(\mathfrak{o}_{2n})$. Specifically, these are highest weight representations, where the highest weight $\mu(u) = (\mu_1(u),\mu_2(u),\ldots)$ is an infinite sequence of $\mu_i(u)\in 1+\CC\dbrack{u^{-1}}$ satisfying \[ \frac{\mu_{i+1}(u)}{\mu_{i}(u)}=\frac{P_i(u+1)}{P_i(u)} \] for monic polynomials $P_i(u)$, and the $\mu_i(u)$ eventually stabilize and equal some even series $\mu^m(u) \in 1+ u^{-2}\CC\dbrack{u^{-2}}$. This implies that the $P_i(u)$ stabilize and equal $1$. As in the integer-rank case, the polynomials $P_i(u)$ are called \textit{Drinfeld polynomials}. Interestingly, the classification has the exact same form for $Y(\xsp_t)$ and $Y(\o_t)$, even though $Y(\xsp_{2n})$ and $Y(\o_{2n})$ have different classification theorems. This is perhaps a consequence of the fact that $\Rep(O_t) = \Rep(Sp_{-t})$ as tensor categories.

The organization of this paper is as follows. In \cref{subsec : classical group rep cats}, we introduce the representations of the groups $GL_N, O_N,Sp_N$ for an even integer $N=2n$, discuss the simple objects and the representations of the Lie algebras. We also discuss the positive-characteristic analogs of these categories, as they will be important later on.

In \cref{sec: Yangians in positive char}, we first define the Yangians and twisted Yangians. We then prove positive-characteristic versions of the classification theorems of finite-dimensional irreducible representations of the twisted Yangians, similarly to how Kalinov in \cite[Section 3]{kalinov_yangians_20} proves a positive-characteristic analog of the classification theorem of finite-dimensional irreducible representations of $Y(\gl_n)$. In \cref{sec: deligne categories}, we define the Deligne categories $\Rep O_t$ and $\Rep Sp_t$ and give their construction using ultraproducts.
Finally, in \cref{sec: yangians in complex rank}, we assemble the results of Sections \ref{sec: Yangians in positive char} and \ref{sec: deligne categories} together using ultraproducts to establish the final classification theorem, \cref{thm: main result}. Finally, \cref{sec: appendix} proves the necessary results about the Sklyanin determinant used for proving \cref{thm: tensor product decomposition}, an important tensor product decomposition that relates representations of $Y(\g_n)$ to those of $SY(\g_n)$ (where $\g_n=\xsp_{2n}$ or $\o_{2n}$), the \textit{special twisted Yangian}, analogous to the Yangian $Y(\xsl_n)$.

Several natural directions remain open. First, one may seek to describe the simple objects of $\Rep(GL_t)$, $\Rep(O_t)$, and $\Rep(Sp_t)$ appearing in irreducible representations of $Y(\mathfrak{gl}_t)$, $Y(\mathfrak{o}_t)$, and $Y(\mathfrak{sp}_t)$. Second, the categorical framework allows one to define a wide class of Yangians and twisted Yangians beyond the classical types, inviting further classification questions. Finally, it remains to be seen whether an analog of Olshanski's centralizer construction exists for twisted Yangians in complex rank, paralleling \cite{utiralova_centralizer_19}.

\subsection*{Acknowledgements}
This paper is the result of MIT PRIMES-USA, a program that provides high-school students an opportunity to engage in research-level mathematics and in which the first author mentored the second author. The authors would like to thank the MIT PRIMES-USA program and its coordinators for providing the opportunity for this research experience. We are grateful to Pavel Etingof for his invaluable insights. We also thank Sidarth Erat for many helpful discussions.

\section{The representation categories \texorpdfstring{$\bfRep(GL_{N})$}{RepGLN}, \texorpdfstring{$\bfRep(SO_{N})$}{RepSON}, \texorpdfstring{$\bfRep(Sp_{N})$}{Rep(SpN)}}\label{classical_groups_rep_cats}

In this section, we review some basic properties about representation categories of classical groups and their Lie algebras. We assume all representations are finite-dimensional, and we work over an algebraically closed field $\KK$ of characteristic $p \neq 2$. We freely reference \cite{jantzen2003representations}.

\subsection{Representation categories of classical groups}\label{subsec : classical group rep cats}
Let $\mathcal G$ be a reductive algebraic group with fixed choice of split maximal torus $H$ and character lattice $X(H)$. Every module over $M$ over $H$ admits a weight-space decomposition $M = \bigoplus_{\lambda \in X(H)} M_\lambda$. In particular, the Lie algebra $\mathfrak g$ of $\mathcal G$ admits a weight-space decomposition given by 

\[\mathfrak g = {\mathfrak g}_0 \oplus \bigoplus_{\alpha \in R} {\mathfrak g}_\alpha,\]
where $R$ is the set of roots, which are the non-zero weights of $H$ such that ${\mathfrak g}_\alpha \neq 0$. We will let $\h$ denote the Cartan subalgebra of $\g$ corresponding to $\g$. Each $f: H \rightarrow \mathbb G_m \in X(H)$ gives a Lie algebra map $\mathfrak h \mapsto \KK \in \h^*$; in particular, if we fix a basis for $X(H)$, the image of the basis under this map gives a basis of $\mathfrak{h}^*$. We will use the same symbols to describe both bases, identified via this map. This will be convenient to pass from weights of $\mathcal G$ to those of $\g$. Now, let's set up notation for the reductive algebraic groups we will consider.
\par
Let $GL_N$ denote the general linear group scheme, which is the linear algebraic group of $N \times N$ invertible matrices, where $N=2n$ is an even integer. The Lie algebra $\gl_N$ of $GL_N$ is the space of all $N \times N$ matrices, for which we will use the usual elementary matrices $E_{ij}$ as basis. However, \textbf{rather than using the usual choice of $1,\ldots, 2n$ to index rows and columns, we will use $I_N \coloneqq \{-n,\ldots,-1,1,\ldots,n\}$}. Consequently, all results we cite from \cite{kalinov_yangians_20} must first be translated to the indices $-n,\ldots, -1,1,\ldots, n$. 

Our choices of split maximal torus $H$ in $GL_N$ and Cartan subalgebra $\h$ in $\mathfrak{gl}_N$ will be the subgroup and subalgebra of diagonal matrices, respectively. We let $\{\eps_i\}_{i\in I_N}$ denote the basis of $X(H)$ where $\eps_i$ is the character $diag(t_{-n},\dots, t_n) \mapsto t_i$. Any weight $\lambda \in X(H)$ is therefore of the form $\lambda = \sum_{i\in I_N}\lambda_i\eps_i$ with $\lambda_i \in \Z$. It follows by our convention that the same set of symbols will denote the basis of $\h^*$ dual to $\{E_{ii}\}_{i \in I_N}$ so that any $\gl_N$-weight $\lambda \in \h^*$ can also be expressed $\lambda = \sum_{i\in I_N}\lambda_i\eps_i$ except now with $\lambda_i \in \KK$. Finally, we will use the conventional triangular decomposition on $\mathfrak g$ where the simple roots are given by $\eps_{i} - \eps_{i+1}$, corresponding to the root vector $E_{i,i+1}$ for $i\in I_N \setminus \{-1\}$, and $\eps_{-1}-\eps_1$, corresponding to the root vector $E_{-1,1}$.
\par
The main groups in this paper, however, will be subgroup schemes of $GL_N$ preserving a form. In particuar let $G$ be an invertible symmetric (resp. skew-symmetric) matrix valued in $\KK$. The \textit{orthogonal group} $O_N(G)$ (resp. \textit{symplectic group} $Sp_N(G)$) is defined to be subgroup scheme given by the relation
\[\{x \in GL_N \ | x^tGx = G\}.\]
Because $O_N(G)$ is not connected, at times it will be convenient to work instead with the special orthogonal group $SO_N(G)$ of determinant $1$ matrices in $O_N(G)$. They both have the same Lie algebra. For the symplectic case we point out that $N$ has to be even for a skew-symmetric form to exist. In fact, from now on we will assume $N = 2n$ is even, as we will not need to work with orthogonal groups for odd $N$.
\par
Because different choices of $G$ yield isomorphic groups, we choose a particular form for $G$ (defined below) and omit it from the notation.  We define the $N\times N$ matrix $G=[g_{ij}]$ by
\begin{align*}
    g_{ij} = \begin{cases}
        \delta_{i,-j} & \text{in the orthogonal case}\\
        \delta_{i,-j}\cdot \sgn i & \text{in the symplectic case.}
    \end{cases}
\end{align*}
The matrix $G$ has the following forms in the orthogonal and symplectic cases, respectively:
\begin{align}\label{choice_of_matrix}
    \begin{bmatrix}
        0 & 0 & \cdots & 0 & 1\\
        0 & 0 & \cdots & 1 & 0 \\
        \vdots & \vdots & \iddots & \vdots &\vdots\\
        0 & 1 & \cdots & 0 & 0 \\
        1 & 0 & \cdots & 0 & 0 \\
    \end{bmatrix},
    \begin{bmatrix}
        0 & 0 & \cdots & 0 & 1\\
        0 & 0 & \cdots & 1 & 0 \\
        \vdots & \vdots & \iddots & \vdots &\vdots\\
        0 & -1 & \cdots & 0 & 0 \\
        -1 & 0 & \cdots & 0 & 0 \\
    \end{bmatrix}.
\end{align}
This is not the most conventional choice (for instance, see \cite{humphreys2012introduction}), but it is what Molev uses in \cite{molev_yangians_book_07}, which will be useful for our purposes. We let the set of diagonal matrices $H$ in $\mathcal G = SO_N, Sp_N$ be our maximal split torus in $\mathcal G$ and is given by matrices of the form $diag(t_{n}^{-1},\dots, t_1^{-1}, t_1, \dots t_n)$. We then define the character $\eps_i \in X(H)$ by sending such a matrix to $t_i$. 
\par
Let's define the Lie algebras of these groups. Define $E$ as the $N\times N$ matrix whose $(i,j)$-entry is $E_{ij} \in \mathfrak{gl}_N$ (we are using $I_N$ to index the rows and columns of $\gl_N$ as well). Next, define the matrix $F$ by:
\begin{align*}
    F \coloneqq E - G^{-1}\E^t G,
\end{align*}
from which a straightforward calculation shows that 
\[F_{ij} = E_{ij}-\theta_{ij}E_{-j,-i},\]
where
\begin{align*}
    \theta_{ij}\coloneqq\begin{cases}
        1 & \text{in the orthogonal case}\\
        \sgn i \cdot \sgn j & \text{in the symplectic case.}
    \end{cases}
\end{align*}
Then, the Lie algebras $\o_N$ and $\xsp$ are the subalgebras of $\gl_N$ spanned the entries of $F$ in the orthogonal case and symplectic case, respectively. The Cartan subalgebra $\mathfrak h$ corresponding to the choice of torus $H$ has basis given by $F_{i,i} = E_{i,i} - \theta_{i,i}E_{-i, -i}$ for $1 \leq i \leq n$; the basis dual to this is precisely the basis induced $\h^*$ by the basis $\{\eps_i\}_{i=1}^n$ of $X(H)$ given above. Finally, if we let the span of the $F_{ij}$ for $1 \leq i < j \leq N$ denote the span of the positive root spaces, then our choice of positive roots $\Phi^+$ is given by

\begin{align}\label{root_system}
    \begin{cases}
        \{-\eps_i -  \eps_j, \eps_i -  \eps_j\}_{i < j} & \mathcal G = SO_N \\
        \{-\eps_i -  \eps_j, \eps_i -  \eps_j\}_{i < j} \cup \{-2\eps_i\} & \mathcal G = Sp_N
    \end{cases},
\end{align}
where $i,j$ range over $1,2,\ldots, n$.
\par
Now, let's discuss a little about the representation theory. Let $\mathcal G$ be any one of $GL_N, SO_N$ or $Sp_N$ with Lie algebra $\mathfrak g$. Call any weight $\lambda \in X(H)$ \textit{dominant integral} if it satisfies  $\langle \lambda, \alpha^\vee\rangle \in \Z_{\geq 0 }$ for all $\alpha \in \Phi^+$, where $\alpha^\vee$ is the coroot associated to $\alpha$, and let $X(H)^+$ be the set of dominant integral weights. Explicitly, this is given by 
\begin{align}\label{explicit_dominant_integral_condition}
    \begin{cases}
        \lambda_{-n} \geq \cdots \geq \lambda_{n} & \mathcal G = GL_N \\
        -|\lambda_1| \geq \cdots \geq \lambda_{n-1} \geq \lambda_{n} & \mathcal G = SO_{N} \\
        0 \geq \lambda_1 \geq \cdots \geq \lambda_{n} & \mathcal G = Sp_{N}
    \end{cases}.
\end{align}
with each $\lambda_i \in \Z$. 
\par
Impose a partial ordering $\preceq$ on $X(H)$ by $\mu \preceq \lambda$ if $\lambda - \mu$ is the nonnegative linear combination of positive roots. Then, it is well-known that the simple modules all have a unique maximal weight with respect to $\preceq$ called the \textit{highest weight}. Moreover, it must be dominant integral and for each dominant integral weight $\lambda$ there is a simple module $\mathcal L(\lambda)$ with that highest weight. In the case of $\mathcal{G} = GL_N$, we will also denote the simple module of highest weight $\lambda$ by $L(\lambda)$ or $L(\lambda, p)$ we if we want to emphasize the characteristic. For $\mathcal G = SO_N, Sp_N$, we will also denote this simple module by $V(\lambda)$ or $V(\lambda, p)$. This will be so that we are consistent with the notations of Molev in \cite{molev_yangians_book_07} later on when discussing Yangians.
\par
We will also consider another class of modules, called \textit{Weyl modules}. For each dominant integral weight $\lambda \in X(H)^+$, we have a Weyl module by $\Delta(\lambda)$ or $\Delta(\lambda, p)$ defined by the universal property that it admits a surjection to any module generated by its highest-weight vector. In the case the characteristic $p = 0$, we have $\mathcal L(\lambda) = \Delta(\lambda)$. For $p > 2$, it can loosely be thought of as a mod $p$-reduction of the characteristic-zero module $\mathcal L(\lambda)$ (so its character is also given by the Weyl character formula). For a more precise definition, see \S II.2 in \cite{jantzen2003representations}.
\par
Let's also state classification simple $O_N$-modules. Let $\mathcal{G} = SO_N$ and recall that $N$ is even. Then, the simple $O_N$-modules are indexed by the following:
\begin{enumerate}
    \item the collection $\{\lambda \ | \ \lambda \in X(H)^+, \lambda_{1} < 0\}$. In particular, the simple modules are given by the induced modules $\mathrm{ind}_{SO_N}^{O_N} \ V(\lambda)$ for $\lambda \in X(H)^+$ with $\lambda_{1} < 0$. Moreover, replacing $\lambda_{1}$ by $-\lambda_{1}$ and then inducing gives the same module;
    \item the collection $\{(\lambda, \varepsilon) \ | \ \lambda \in X(H)^+, \lambda_{1} = 0, \varepsilon \in \{-1,1\}\} $. In this case, the $SO_N$-module structure on $V(\lambda)$ lifts to the structure of an $O_N$-module where a distinguished order-$2$ element acts by $\varepsilon$. 
\end{enumerate}
When it is clear we are working with $O_N$, we will let $X(H)^+$ for $O_N$ be the set of $\lambda \in X(H)$ such that $0 \geq \lambda_1 \geq \cdots \geq \lambda_n$ and denote the simple module corresponding to $\lambda$ with the notation $V(\lambda) $ or $V(\lambda, p)$. As written, this is ambiguous when $\lambda_1 = 0$, but we deliberately ignore the sign $\varepsilon$ for the case $\lambda_1 = 0$ because it is not detectable at the level of the Lie algebra $\o_N$ and basically has no bearing on what is to come. Moreover, it will turn out we won't really be interested in the case where $\lambda_1 \neq 0$ due to stabilization with respect to a certain limit as $N\rightarrow \infty$, so in a loose sense we truly only care about the underlying $SO_N$-module structure.
\par
We will not prove this classification here, but a proof is given in \cite{goodman2000representations} (the proof therein assumes characteristic $p = 0$ but only needs $p \neq 2$).
\begin{rem}\label{rem : odd vs even orthogonal groups}
    For $\mathcal{G} = SO_N$ with $N = 2n + 1$ odd, the root system for $\mathcal{G}$ is the same as the even case except with the additional short roots of $\eps_i$ and $-\eps_i$. The criteria for dominant integrality for $\lambda \in X(H)$ then becomes the same as that for $Sp_N$. Since we have $O_N = SO_N \times \Z/2\Z$, it is clear that $O_N$-representations are indexed by the set $\{(\lambda, \varepsilon) \ | \lambda \in X(H)^+, \varepsilon \in \{-1, 1\}\}$.
\end{rem}

\subsection{Representations of classical Lie algebras} Continue to let $\mathcal G$ be one of $GL_N, SO_{N}$ or $Sp_{N}$, and let $\g$ denote the Lie algebra of $\mathcal G$. We are working towards describing the representation theory of the Yangian $Y(\mathfrak{g})$, so it will also be useful to state some basic facts about the representation theory of $\mathfrak g$. In characteristic zero, restricting a representation of $\mathcal G$ to $\mathfrak g$ is essentially the same, so we are mostly concerned with positive characteristic. 
\par
Let $C_{\Z}$ denote the \textit{fundamental alcove} of $\mathcal G$, which is given by 
\[C_\Z = \{\lambda \in X(H) \ | \ 0 \leq \langle \lambda + \rho, \alpha^\vee\rangle < p', \forall \alpha \in \Phi^+\},\] 
where $p' = p$ if $p > 0$ and $p' = \infty$ if $p = 0$ and where $\rho$ is half the sum of the positive roots. For $O_N$, we define its fundamental alcove to be that for $SO_N$. Then, we have the following well-known fact (cf. \S II.5.6 in \cite{jantzen2003representations}):

\begin{prop}\label{prop : fundamental alcove has no extensions and is simple}
    For any $\lambda,\mu \in C_{\Z} \cap X(H)^+$, we have $\mathcal L(\lambda) = \Delta(\lambda)$, and moreover there are no non-trivial extensions between $\mathcal L(\lambda)$ and $\mathcal L(\mu)$.
\end{prop}
\begin{cor}
    For $\lambda \in C_\Z \cap X(H)^+$, the  $\mathcal G$-module $\Delta(\lambda)$ is irreducible upon restriction to $\mathfrak{g}$.
\end{cor}
\begin{proof}
    If $p = 0$, this is well-known. So suppose $p > 0$. First of all, by Proposition \ref{prop : fundamental alcove has no extensions and is simple}, $\Delta(\lambda)$ is simple. Then, because $\mathcal G$ is reductive, the restriction of a simple $\mathcal G$-module to $\mathfrak g$ can instead be thought of as module over the first Frobenius kernel over $\mathcal G$. But $\lambda$ is a restricted dominant integral weight because it is both dominant integral and in the fundamental alcove, so the claim follows by the Steinberg tensor product theorem (see \S II.3.17 in \cite{jantzen2003representations}).
\end{proof}
The following lemma will also be useful (recalling $N = 2n$) :
\begin{lem}\label{lem: irreducible g_n-modules in positive char}
    Suppose $p > 2$, and let $\lambda$ be a dominant integral weight for $\mathcal G = SO_N, Sp_N$. Then, the condition $-2\lambda_n + 2n < p$ implies $\lambda \in C_{\ZZ}$.
\end{lem}
\begin{proof}
We want to show that $-2\lambda_n+2n<p$ implies \[0 \le \langle\lambda + \rho, \alpha^{\vee}\rangle < p\] for every positive root $\alpha$. Let us first consider the case of $\mathcal G = SO_N$, whose positive roots are $\pm \eps_i - \eps_j$ for $i<j$. In this case, we have 
\[\rho = -\eps_2 - 2\eps_3 -\cdots -(n-1)  \eps_{n}.\]
The condition that $\lambda$ lies in the fundamental alcove is then equivalent to:
\begin{align*}
   &0 \le \langle\lambda + \rho, (\pm \eps_i - \eps_j)^\vee\rangle = (\pm\lambda_i - \lambda_j) \mp (i-1) + (j-1) < p .
\end{align*}
Since $\lambda$ is dominant integral, $i<j$ means $0\ge \lambda_i\ge \lambda_j$, so the first inequality is satisfied. Additionally, $\pm \lambda_i-\lambda_j \le -2\lambda_n$ and $\mp (i-1)+(j-1)\le 2n$, so if $-2\lambda_n+2n<p$, the second is satisfied, so $-2\lambda_n+2n<p$ implies $\lambda$ lies in the fundamental alcove. The same argument applies, with minor adjustments, to the case of $\mathcal G = Sp_N$, whose positive roots are $\pm\eps_i- \eps_j$ for $i<j$ and $-2\eps_i$, and
\[\rho = -\eps_1 -2\eps_2 \cdots -n\eps_n.\]
This completes the proof.
\end{proof} 

\subsection{Representation categories}\label{subsec: rep cats}
We would like to define some useful notation for later on. Let $\mathcal G$ be a reductive algebraic group with Lie algebra $\g$. If $A$ is some Lie algebra, affine group scheme, or associative algebra algebra (in the category of vector spaces), let $\bfRep A$ denote the category of finite-dimensional $A$-modules. For instance, we let $\bfRep \mathcal G$ denote the category of finite-dimensional representations of $\mathcal G$, which obviously is a symmetric tensor category. In characteristic $0$, it is semisimple.  
\par
Now, let $A$ be an associative unital algebra equipped with an algebra homomorphism $U(\g) \rightarrow A$. Let us define $\bfRep_{\mathcal G} A$ to be the category of $\mathcal G$-modules $M$ that also admit the structure of an $A$-module such that the action of $\g$ on $M$ through $\mathcal G$ and through $A$ coincide (we will treat the $\mathcal G$-module admitting multiple such structures as separate objects). The morphisms will be $\mathcal G$-module homomorphisms. For instance, if $A = U(\g)$ and the equipped algebra homomorphism is the identity map, then this category is just the category $\bfRep \mathcal G$. We will particularly be interested in the case where $A$ is one of the (twisted) Yangians $Y(\gl_N)$, $Y(\o_N)$, or $Y(\xsp_N)$, defined in \S \ref{sec: Yangians in positive char}. We do so because we are interested only in representations that lift to the group, after which we will take ultraproducts to construct representations in complex rank.

\section{Twisted Yangians}\label{sec: Yangians in positive char}
Before we can develop the theory of twisted Yangians in Deligne's categories, we would like to first discuss twisted Yangians in finite rank. We will review basic results in characteristic zero. It will also be necessary to prove some basic results about twisted Yangians in positive characteristic due to a limit construction later on that passes from fields of positive characteristic to the complex numbers using ultrafilters. \footnote{Discussion of twisted Yangians in positive characteristic appears to be virtually nonexistent (cf. \cite{brundan_topley_modular_yangians_18} for the ordinary Yangian of $\gl_N$), so we prove what we need here.} We closely follow the presentation in \cite{molev_yangians_book_07}. Retain the notation from \cref{classical_groups_rep_cats}. In particular, we work over the algebraically closed field $\KK$ of characteristic $p \neq 2$ and have $N = 2n$ be some even positive integer.

\subsection{The (twisted) Yangians \texorpdfstring{$Y(\gl_{N})$, $Y(\xsp_{N})$, $Y(\o_{N})$}{Y(gl2n), Y(sp2n), Y(o2n)}}\label{subsec: twisted yangians in integer rank} 

We recall the definitions of the Yangian of $\gl_{N}$ and the twisted Yangians of $\o_{N}$ and $\xsp_{N}$, following Chapter 4 of \cite{molev_yangians_book_07}, as well as some basic properties.

The Yangian $Y(\gl_N)$ is the associative algebra generated by the generators generators $t_{ij}^{(r)}$ for $r\ge 1$ and $i,j\in I_N$ subject to the well-known ternary relation (see \cite{molev_yangians_book_07} or \eqref{eqn: ternary relation}). It will be useful to consolidate these and define the formal power series \[t_{ij}(u) \coloneqq \delta_{ij}+t_{ij}^{(1)}u^{-1}+t_{ij}^{(2)}u^{-2}+\cdots\]
and also to further assemble these into the $N\times N$ matrix $T(u) = [t_{ij}(u)]$, whose coefficients can be viewed as elements of  \[\Hom(V, V\otimes Y(\gl_N))\dbrack{u^{-1}},\]
where $V=\KK^N$.     It can be thought of as a generalization of the universal enveloping algebra of $U(\gl_N)$.
\begin{prop}\label{prop: eval homomorphism and embedding glN}
        The map
        \begin{align}\label{eqn: evaluation homomorphism glN}
            t_{ij}(u) \mapsto \delta_{ij} + E_{ij}u^{-1}
        \end{align}
        defines a surjective homomorphism \[\pi_N: Y(\gl_N)\to U(\gl_N).\] Moreover, we have an embedding $i_n: U(\gl_N)\hookrightarrow Y(\gl_N)$ given by
        \begin{align}\label{eqn: embedding glN}
            E_{ij}\mapsto t_{ij}^{(1)}.
        \end{align}
\end{prop}
Here $E_{ij}$ denotes the elementary $N \times N$ matrix in $\gl_N$ which has a $1$ in the $(i,j)$ position and zeros elsewhere. Moreover, $Y(\gl_N)$ is a Hopf algebra and the maps above are compatible with the Hopf algebra structure on $U(\gl_N)$. A useful power series is the \textit{quantum determinant}, given by 
\begin{align}\label{eqn: qdet}
    \qdet T(u) \coloneqq \sum_{\sigma \in \mfS_N} \sgn(\sigma)\cdot t_{1,\sigma(1)}(u-N+1)\cdots t_{n,\sigma(n)}(u).
\end{align}
In characteristic zero, the coefficients $d_1,d_2,\ldots$ of the quantum determinant are algebraically independent and generate the center of $Y(\gl_N)$. In positive characteristic, this is not true. However, as seen in \cite{kalinov_yangians_20}, the coefficients are still central and algebraically independent; hence we can define $Z_{HC}(Y(\gl_N))$ to be the subalgebra generated by the coefficients of $\qdet T(u)$.
It will also be useful to work with the Yangian $Y(\xsl_N)$ of $\xsl_N$. It is defined as follows. It can be checked that for any power series $f(u) \in 1 + u^{-1}\KK\dbrack{u^{-1}}$, the map 
\begin{align}\label{eqn: power series automorphism}
    T(u)\mapsto f(u)T(u)
\end{align}
defines an automorphism of $Y(\gl_N)$. The Yangian $Y(\xsl_N)$ of $\xsl_N$ is defined as the subalgebra of $Y(\gl_N)$ invariant under all such automorphisms. It was shown in \cite{brundan_topley_modular_yangians_18} that
\begin{align}\label{eqn: Y(gl_N) tensor product decomposition}
Y(\gl_N) = Z_{HC}(Y(\gl_N))\otimes Y(\xsl_N).
\end{align}
\par
Now, let us define the twisted Yangians $Y(\mathfrak{o}_{N})$ and $Y(\xsp_{N})$. Since we will be dealing with both cases simultaneously, we let $\g_n = \mathfrak{o}_N$ or $\xsp_N$.
Recall our choice of the matrix $G$ from \ref{choice_of_matrix}, and that we are using $I_N = \{-n,\ldots, -1,1,\ldots, n\}$ to index the rows and columns. For an $N \times N$ matrix $A$, we define the transpose operation $'$ with respect to $G$ by
\begin{align}\label{eqn: transpose wrt form}  
    A' \coloneqq G^{-1} A^t G.  
\end{align}  
In contrast, \cite{molev_yangians_book_07} defines $A'$ as $GA^t G^{-1}$. Our choice will make it easier to state definitions categorically later on, as it coincides with the transpose with respect to the bilinear form on $V$ induced by $G$. However, since our choice of $G$ satisfies $G = \pm G^{-1}$, the two definitions are equivalent. Consequently, all results we cite from \cite{molev_yangians_book_07} remain valid in our setting. If $A=[a_{ij}]$ and $A' = [a'_{ij}]$, we can compute:
\begin{align*}
    a_{ij}' = \theta_{ij}a_{-j,-i}.
\end{align*}
Then $\F = \E - \E'$. We defined $\g_n$ to be the algebra spanned by the $F_{ij}$, but we can alternatively describe it as the set of all all $X\in\gl_N$ satisfying $X'=-X$. 

Next, define the $R$-matrix $R(u)$, an $N\times N$ matrix, by: \[R(u)\coloneqq 1-\frac{P_n}{u} \in R(u)\in (\End(V)\otimes \End(V))\dbrack{u^{-1}},\] where $P_n$ is the $N\times N$ matrix given by:
\[P_n= \sum_{i,j}E_{ij}\otimes E_{ji}.\]
Note that $P_n$ is the flip map.
The transposed $R$-matrix is then defined as: \[R'(u)\coloneqq G_1^{-1}R^{t_1}(u)G_1=G_2^{-1}R^{t_2}(u)G_2,\] where $G_i$ and $t_i$ act on the $i$-th tensor component for $i=1$, $2$.

Similarly, we define the twisted Yangian by introducing a family of generators $\s_{ij}^{(r)}$ and the formal series
    \[\s_{ij}(u) \coloneqq \delta_{ij}+ \s_{ij}^{(1)}u^{-1}+\s_{ij}^{(2)}u^{-2}+\cdots .\]
   We then assemble these into the $N\times N$ matrix $\mcS(u)$ with entries $\s_{ij}(u)$. Writing \[\mcS(u) =\sum_{i=0}^\infty \mcS_i u^{-i},\] we note that $\mcS_0$ is simply the identity matrix.

\begin{defn}\label{defn: twisted yangian}
        The twisted Yangian $Y(\g_n)$ is the associative unital algebra generated by the elements $\s_{ij}^{(r)}$, with defining relations given by
        \begin{align}\label{eqn: quaternary relation}
            R(u-v)\mcS_1(u)R'(-u-v)\mcS_2(v) &= \mcS_2(v)R'(-u-v) \mcS_1(u) R(u-v) \\
        \label{eqn: symmetry relation}
            \mcS'(-u) &= \mcS(u) \pm \frac{\mcS(u)-\mcS(-u)}{2u}.
        \end{align}
\end{defn}
We call the first relation the \emph{quaternary} relation and the second the \emph{symmetry} relation.
The following standard results (again, proved in the same way as characteristic zero; see \cite[Chapter 2]{molev_yangians_book_07}) will be useful.
\begin{prop}\label{prop: eval homomorphism and embedding}
    We define the evaluation homomorphism $Y(\g_n)\to U(\g_n)$ by
    \begin{align}\label{eqn: eval homomorphism}
        \varrho_n\colon \s_{ij}(u)\mapsto \delta_{ij} + F_{ij}\left(u\pm \frac{1}{2}\right)^{-1}.
    \end{align}
    Conversely, the embedding $U(\g_n)\hookrightarrow Y(\g_n)$ is given by
    \begin{align}\label{eqn: embedding homorphism}
        i_n\colon F_{ij}\mapsto \s_{ij}^{(1)}.
    \end{align}
\end{prop}
\begin{prop}
    For any series \[g(u)=1+g_1u^{-2}+g_2u^{-4} + \cdots\in 1+u^{-2}\KK\dbrack{u^{-2}},\] the mapping
    \begin{align}\label{eqn: twisted power series automorphism}
        \mcS(u)\mapsto g(u)\mcS(u)
    \end{align}
    defines an automorphism of $Y(\g_n)$.
\end{prop}
There is a natural inclusion $Y(\g_n)\hookrightarrow Y(\gl_N)$ that aligns with the embedding $\g_n\hookrightarrow \gl_N$. This inclusion is given by
\begin{align}\label{eqn: embed Y(gn) into Y(glN)}
        \mcS(u) \mapsto T(u)T'(-u).
\end{align}

We define the \textit{special twisted Yangian} $SY(\g_n)$ as the subalgebra of $Y(\g_n)$ that remains invariant under all automorphisms of the form \eqref{eqn: twisted power series automorphism}. Equivalently, it can be characterized as \[SY(\g_n)=Y(\xsl_N)\cap Y(\g_n).\]

We define the scalar $\alpha_q(u)$ as follows:
\begin{align}\label{eqn: alphaq(u)}
    \alpha_q(u) = \begin{cases}
        1 &\text{in the orthogonal case}\\
        \frac{2u+1}{2u-q+1} & \text{in the symplectic case}.
    \end{cases}
\end{align}

Next, define the special map $\varphi_N: \mfS_N\to\mfS_N$, which acts on permutations as in \cite[Section 2.7]{molev_yangians_book_07}. Let $(a_1,\ldots, a_N)$ be a permutation of $I_N$.

Define 
\begin{align}\label{eqn: sdet}
    \begin{split}
    \sdet \mcS(u) &\coloneqq (-1)^n\alpha_N(u)\\&\times \sum_{\sigma \in \mfS_N}\sgn\sigma\sigma' \cdot \s'_{-a_{\sigma(1)}, a_{\sigma'(1)}}(-u)\cdots \s'_{-a_{\sigma(n)},a_{\sigma'(n)}}(-u+n-1) \\
    &\hspace{20mm} \times \s_{-a_{\sigma(n+1)},a_{\sigma'(n+1)}}(u-n)\cdots \s_{-a_{\sigma(N)},a_{\sigma'(N)}}(u-N+1)
    \end{split}
\end{align}
to be the \textit{Sklyanin determinant}. In characteristic zero, the coefficients of the Sklyanin determinant generate the center of $Y(\gl_N)$. In positive characteristic, this is not true. However, as we will see below, they are still central, and hence we can define $Z_{HC}(Y(\gl_N))$ as the algebra generated by the coefficients of $\sdet \mcS(u)$. Write
\[\sdet \mcS(u) = 1+ c_1 u^{-1}+c_2u^{-2}+\cdots.\]
\begin{thm}\label{thm: tensor product decomposition}
    There exists a series $\sdet\mcS(u)\in 1+ u^{-1}Y(\g_n)\dbrack{u^{-1}}$ such that the coefficients of $\sdet \mcS(u)$ are central, and the even coefficients $c_2,c_4,\ldots$ are algebraically independent and generate $Z_{HC}(Y(\g_n))$; so $Z_{HC}(Y(\g_n))=\KK[c_2,c_4,\ldots]$. 
    We have the following tensor product decomposition:
    \begin{align}\label{eqn: tensor product decomposition}
        Y(\g_n) = Z_{HC}(Y(\g_n))\otimes SY(\g_n).
    \end{align}
\end{thm}
See \cref{sec: appendix} for the proof. These results are all well-known in characteristic zero; see \cite{molev_yangians_book_07}. To prove the first statement in positive characteristic, we simply repeat the same arguments as in characteristic zero. To prove the tensor product decomposition, we use \eqref{eqn: Y(gl_N) tensor product decomposition}. We do not show that the series $\sdet \mcS(u)$ in \cref{thm: tensor product decomposition} matches \eqref{eqn: sdet} as it is not important for our purposes. 
\subsection{Representations of \texorpdfstring{$Y(\g_n)$}{Y(gn)}}
Let $\mathcal G_n$ denote the algebraic group $O_N$ or $Sp_N$ such that its Lie algebra is $\g_n$. Molev derives the classification finite-dimensional simple modules over $Y(\g_n)$ in \cite{molev_yangians_book_07} in characteristic zero. The goal of this subsection is to extend the results therein to positive characteristic, imposing assumptions as needed. Due to the ultraproduct construction later on, we mainly only care about representations detectable at the level of $\mathcal G$, so we will primarily work with representations in the category $\bfRep_{\mathcal G_n} Y(\g_n)$ (recall the notation from \S \ref{subsec: rep cats}). 
\par
Finally, we will need some facts about the representation theory of $Y(\gl_N)$ in positive characteristic. For the definition and theory of highest weight modules for $Y(\gl_N)$ in positive characteristic, see \cite{kalinov_yangians_20} and \cite{brundan_topley_modular_yangians_18}.
\begin{defn}
    A representation $V$ of the twisted Yangian $Y(\g_n)$ is called a \emph{highest weight representation} if there exists a nonzero vector $\xi\in V$ such that $V$ is generated by $\xi$ and the following condition holds: for each $i$, there is a formal power series
\begin{align} \mu_i(u) = 1 + \mu_i^{(1)} u^{-1} + \mu_i^{(2)} u^{-2} + \cdots, \quad \mu_i^{(r)} \in \mathbb{K}, \end{align}
such that
\begin{align} \s_{ij}(u) \xi &= 0 \quad \text{for } i < j , \\ \s_{ii}(u) \xi &= \mu_i(u) \xi \quad \text{for }  1 \leq i \leq n. \end{align}
    The vector $\xi$ is called the \emph{highest weight vector} of $V$, and the tuple $\mu(u) = (\mu_1(u),\ldots,\mu_n(u))$ is called the \emph{highest weight}.
\end{defn}
Notice that the symmetry relation gives
\begin{align}\label{eqn: eigenvector for s2n-i,2n-i}
    \s_{-i,-i} = \s_{ii}(-u)\pm \frac{\s_{ii}(u)-\s_{ii}(-u)}{2u}
\end{align}
for $i=1,\ldots, n$.
Then \eqref{eqn: eigenvector for s2n-i,2n-i} implies that $\xi$ is an eigenvector for $\s_{-i,-i}$ as well. 
\begin{defn}
    Let $\mu(u) = (\mu_1(u),\ldots,\mu_n(u))$. The \emph{Verma module} $M(\mu(u))$ over $Y(\g_n)$ is the quotient of $Y(\g_n)$ by the left ideal generated by the coefficients of the series 
    \[\begin{cases}
        \s_{ij}(u) & i < j \\
        \s_{ii}-\mu_i(u) & i=1,\ldots,n .
    \end{cases}\] 
    A standard argument shows that $M(\mu(u))$ has a unique simple quotient, which we will denote by $V(\mu(u))$, and also has the universal property that it surjects onto any highest weight module with the same highest weight.  
\end{defn}
Recall the notation for simple $\mathcal G_n$ and $GL_N$-modules from \S \ref{subsec : classical group rep cats}. We have the following proposition:
\begin{prop}\label{prop: positive char highest weight}
    Let $V \in \bfRep_{\mathcal G_n} Y(\g_n)$ be simple and such that $V = V(\lambda_1,p)\oplus \cdots \oplus V(\lambda_k,p)$ for $\lambda_i \in X(H)^+ \cap C_\Z$ with $(\lambda_i)_1 = 0$. Then, the module $V$ has a unique singular vector (up to scaling), whose $\g_n$ weight is maximal.
\end{prop}
\begin{proof}
    Let us first contextualize the assumptions. First of all, the condition $(\lambda_i)_1 = 0$ is only there to ensure that each $\mathcal G_n$-module appearing above is irreducible when restricting to $SO_N$ and therefore to $\mathfrak{o}_N$, when $\mathcal G_n = O_N$. The condition that $\lambda_i \in X(H)^+ \cap C_\Z$ is so that the character of $V(\lambda_i, p)$ is given by the Weyl character formula, which loosely speaking is to ensure that for a given $\lambda_i$, the module $V(\lambda_i, p)$ ``looks the same" as $p$ varies (this will be relevant when working with ultraproducts).
    \par
    Now, let's prove the statement. The problem here is to rule out the possibility of $V$ not having any singular vectors at all. Since we can regard $V$ as a $\g_n$-module, we can repeat verbatim the argument given in Theorem 4.2.6 of \cite{molev_yangians_book_07}.
\end{proof}
Hence $V$ must be a highest weight representation $V=V(\mu(u))$ by the universal property of $M(\mu(u))$. We also get the following corollary (Corollary 4.2.7 in \cite{molev_yangians_book_07}).
\begin{cor}\label{cor: positive char highest weight vector}
    If $V$ satisfies the conditions of \cref{prop: positive char highest weight}, and $\eta\in V$ with $\s_{ij}(u)\eta = 0$ for all $i<j$, then $\eta$ is a scalar multiple of the highest weight vector $\xi$ of $V$.
\end{cor}
Now, we will use what is known about $Y(\gl_N)$ to say more about representations of $Y(\g_n)$. As noted at the start of Section 3 in \cite{kalinov_yangians_20}, all the statements in Proposition 2.1.2 of \cite{kalinov_yangians_20} except for (d) hold in any characteristic. In particular, $Y(\gl_N)$ is still a Hopf algebra. Moreover, recall that we can view $Y(\g_n)$ as a subalgebra of $Y(\gl_N)$ via the embedding
\[\mcS(u)=T(u)T'(-u).\]
Letting $\Delta$ denote the comultiplication on $Y(\gl_N)$, the following formula from \cite[Section 4.2]{molev_yangians_book_07}
\begin{align}\label{eqn: comultiplication on sij(u)}
    \Delta(\s_{ij}(u)) = \sum_{a,b}\theta_{bj} t_{ia}(u) t_{-j,-b}(-u) \otimes \s_{ab}(u)
\end{align}
holds in any characteristic. This gives us the following proposition:

\begin{prop}\label{prop: functor between rep cats of Yangians}
    There is a functor $\bfRep_{GL_N} Y(\gl_N) \boxtimes \bfRep_{\mathcal G_n} Y(\g_n) \rightarrow \bfRep_{\mathcal G_n} Y(\g_n)$ induced by $\Delta$.
\end{prop}
\begin{proof}
    Notice that $\Delta$ is an algebra homomorphism that maps $Y(\g_n)$ to $Y(\gl_N)\otimes Y(\g_n)$. Moreover it sends $U(\g_n)$ to $U(\g_n) \otimes U(\g_n)$. It follows that if $M \in \bfRep_{GL_N} Y(\gl_N)$ and $M' \in M$, then $M \otimes M'$ is a $\mathcal G_n$ module where both $\g_n$ actions are compatible.
\end{proof}
Let $\zeta$ be the highest weight vector of the $Y(\gl_N)$-module $L(\lambda(u))$ (see \cite{kalinov_yangians_20} for the definition of a $Y(\gl_N)$-highest weight module), where the highest weight is given by $\lambda(u)=(\lambda_1(u),\ldots,\lambda_N(u))$. Similarly, let $\xi$ denote the highest weight vector of the $Y(\g_n)$-module $V(\mu(u))$. 
\begin{prop}\label{prop: highest weight of tensor product module}
    The submodule $Y(\g_n)(\zeta\otimes \xi)$ of the $Y(\g_n)$-module $L(\lambda(u))\otimes V(\mu(u))$ over $Y(\g_n)$ is a highest weight representation with highest weight vector $\zeta\otimes \xi$. Moreover, the highest weight is $\lambda_i(u)\lambda_{-i}(-u)\mu_i(u)$, for $i=1,\ldots, n$.
\end{prop}
\begin{proof}
    The proof is the same as the characteristic zero case; see \cite[Proposition 4.2.9]{molev_yangians_book_07}.
\end{proof}
The evaluation homomorphism \eqref{eqn: eval homomorphism} gives a functor $\bfRep {\mathcal G_n}\rightarrow \bfRep_{\mathcal G_n} Y(\g_n)$. Similarly, there is a functor $\bfRep{GL_N}\rightarrow \bfRep_{ GL_N} Y(\gl_N)$.  Let us call the image of any module under these functors an \textit{evaluation module}. We will use the same notation and not distinguish between a module and its image under this functor.
\par
Let $L(\lambda^{(1)}),\ldots, L(\lambda^{(k)})$ be irreducible highest weight $GL_N$-modules such that each $\lambda^{(i)}$ is dominant integral and in the fundamental alcove for all $i$, and let $V(\mu)$ be an irreducible $\mathcal G_n$-module with $\mu$ dominant integral and such that $\mu_1 = 0$ and also in the corresponding fundamental alcove. View these as evaluation modules, and consider their tensor product
\begin{align}\label{eqn: tensor product of eval modules}
    L(\lambda^{(1)})\otimes \cdots \otimes L(\lambda^{(k)}) \otimes V(\mu),
\end{align}
which lies in $\bfRep_{\mathcal G_n} Y(\g_n)$ by Proposition \ref{prop: functor between rep cats of Yangians}. Define
\begin{align}\label{eqn: highest weight of Y(glN)-evaluation module}
    \lambda_i(u) = (1+\lambda_i^{(1)}u^{-1})\cdots (1+\lambda_i^{(k)}u^{-1}).
\end{align}
Additionally, using the formula \eqref{eqn: eval homomorphism}, we can get that the $Y(\g_n)$-module $V(\mu)$ is a highest-weight representation, with highest weight
\begin{align}\label{eqn: highest weight of Y(gn)-evaluation module}
    \mu_i(u) = \frac{1+(\mu_i\pm \tfrac 12)u^{-1}}{1\pm \tfrac 12u^{-1}}, \qquad i=1,\ldots, n
\end{align}
Let $\zeta$ be the tensor product of all the highest weight vectors of the modules in \eqref{eqn: tensor product of eval modules}. Then we have the following proposition.

\begin{prop}\label{prop: highest weight of tensor product of eval modules}
    The submodule $Y(\g_n)\zeta$ of the $Y(\g_n)$-module \eqref{eqn: tensor product of eval modules} is a highest weight representation with the highest weight vector $\zeta$. Moreover, the $i$-th component of the highest weight is given by $\lambda_i(u)\lambda_{-i}(-u)\mu_i(u)$, where $\lambda_i(u)$ is given by \eqref{eqn: highest weight of Y(glN)-evaluation module} and $\mu_i(u)$ is given by \eqref{eqn: highest weight of Y(gn)-evaluation module}.
\end{prop}
\begin{proof}
    The proof is the same as the characteristic zero case; see \cite[Proposition 4.2.11]{molev_yangians_book_07}
\end{proof}
We can also consider modules of the form
\begin{align}\label{eqn: tensor product of eval modules 2}
    L(\lambda^{(1)})\otimes \cdots \otimes L(\lambda^{(k)}).
\end{align}
Let $\zeta$ be the tensor product of all the highest weight vectors of modules in \eqref{eqn: tensor product of eval modules 2}. Then we have the following immediate corollary of \cref{prop: highest weight of tensor product of eval modules}. 
\begin{cor}\label{cor: highest weight of tensor product of eval modules}
    The submodule $Y(\g_n)\zeta$ of the $Y(\g_n)$-module \eqref{eqn: tensor product of eval modules 2} is a highest weight representation with the highest weight vector $\zeta$. Moreover, the $i$-th component of the highest weight is given by $\lambda_i(u)\lambda_{-i}(-u)$.
\end{cor}
\begin{prop}\label{prop: sij(u) are polynomial operators} The following holds:
\begin{enumerate}[$(i)$]
    \item For any element $\eta$ of the $Y(\g_n)$-module \eqref{eqn: tensor product of eval modules} and indices $i,j$, the expression \[\left(1\pm \frac 12u^{-1}\right)\s_{ij}(u)\eta\] is a polynomial in $u^{-1}$ of degree at most $2k+1$.
    \item For any element $\eta$ of the $Y(\g_n)$-module \eqref{eqn: tensor product of eval modules 2}, the expression $\s_{ij}(u)\eta$ is a polynomial in $u^{-1}$ of degree at most $2k$.
\end{enumerate}
\end{prop}
\begin{proof}
    Once again, the proof is the same as the characteristic zero case; see Propositions 4.2.13 and 3.2.11 in \cite{molev_yangians_book_07}.
\end{proof}
The following is a consequence of \eqref{eqn: tensor product decomposition}. 

\begin{prop}\label{prop: representations of Y(gn) and SY(gn) positive char}
    Every finite-dimensional irreducible representation of the twisted Yangian $Y(\g_n)$ remains irreducible when restricted to the special twisted Yangian $SY(\g_n)$. Moreover, every finite-dimensional irreducible of $SY(\g_n)$ is of this form.
\end{prop}

We write $\nu(u)\longrightarrow \mu(u)$ if there exists a monic polynomial $P(u)$ such that $\frac{\nu(u)}{\mu(u)}=\frac{P(u+1)}{P(u)}$.

\begin{prop}\label{prop: gln condition holds for gn}
    If  $V(\mu(u))$ is finite-dimensional and satisfies the conditions of \cref{prop: positive char highest weight}, then 
    \begin{align}\label{eqn: gln condition holds for gn}
        \mu_1(u)\longrightarrow \mu_2(u)\longrightarrow \cdots \longrightarrow \mu_n(u).
    \end{align}
\end{prop}
\begin{proof}
        The proof closely parallels Proposition 4.2.8 in \cite{molev_yangians_book_07}. Let $J$ be the left ideal of $Y(\g_n)$ generated by the coefficients of $\s_{-i,j}(u)$, for $i,j=1,\ldots, n$, and consider 
    \begin{align*}
        V^J = \{\eta\in V(\mu(u)) \mid \s_{-i,j}(u)\eta = 0 \text{ for all } i,j \in\{1,\ldots,n\}\}.
    \end{align*}
    The highest weight vector $\xi \in V(\mu(u))$ belongs to $V^J$. The defining relations give
    \[
    [\s_{-i,j}(u),\s_{kl}(v)] = 0 \mod J,
    \] 
    for $i,j,k,l>0$, so the operators $\s_{kl}(u)$ with $k, l>0$ preserve $V^J$. Moreover for $i,j,k,l>0$ the defining relations give
    \[
    (u-v)[\s_{ij}(u),\s_{kl}(v)] = \s_{kj}(u)\s_{il}(v) - \s_{kj}(v)\s_{il}(u) \mod J,
    \] which is formally identical to that of $Y(\gl_n)$, so $V^J$ admits a $Y(\gl_n)$-action by $t_{ij}(u) \mapsto \s_{ij}(u)$. The cyclic span $L\coloneqq Y(\gl_n)\xi$ is thus a finite-dimensional highest weight module of $Y(\gl_n)$ with weight $(\mu_1(u),\ldots,\mu_n(u))$.

Decomposing $V(\mu(u)) =\bigoplus_i V(\mu_i,p)$ as a $\g_n$-module, and recalling that the Weyl group stabilizes the weights, yields that for any $\gl_{2n}$-weight $\eta$ appearing in $V(\mu(u))$, \[ (\eta)_{-n}-(\eta)_n +n \le 2\max_i-(\mu_i)_n+n<p. \] This verifies the condition of Proposition 3.2.1 in \cite{kalinov_yangians_20}. It follows from Theorem 3.2.5 in \cite{kalinov_yangians_20} that \eqref{eqn: gln condition holds for gn} holds.
\end{proof}

We now split into the separate cases of $\xsp_{2n}$ and $\o_{2n}$.

\subsubsection{Finite-dimensional irreducible representations of $Y(\xsp_{2n})$ in positive characteristic} 
\begin{prop}\label{prop: sp2 highest weight representation positive char}
    Suppose $V$ is irreducible and finite-dimensional, and as a $\xsp_2$-representation is isomorphic to $V(\mu_1,p)\otimes \cdots \otimes V(\mu_k,p)$ with each $\mu_i$ is a non-positive integer and $p>2\max -\mu_i + 2$. Then $V$ is a highest representation $V=V(\mu(u),p)$.
\end{prop}
\begin{proof}
    The condition implies that $\mu_i$ is in the fundamental alcove, and hence each $V(\mu_i,p)$ is irreducible and highest weight. Then we can repeat the proof of Theorem 4.2.6 in \cite{molev_yangians_book_07}.
\end{proof}
From now on, assume $V$ satisfies the conditions of \eqref{prop: sp2 highest weight representation positive char}.
Write
\begin{align}
    \mu(u) = 1+\mu^{(1)}u^{-1}+\mu^{(2)}u^{-2}+\cdots,\quad \mu^{(r)}\in \KK.
\end{align}
\begin{prop}\label{prop: Y(sp2) polynomial weights}
    There exists an even formal series $g(u)\in 1+u^{-2}\KK\dbrack{u^{-2}}$ such that $g(u)\mu(u)$ is a polynomial in $u^{-1}$.
\end{prop}
\begin{proof}
    The proof of Proposition 4.3.1 in \cite{molev_yangians_book_07} carries over verbatim.
\end{proof}
Then composing the action of $Y(\xsp_2)$ on $V(\mu)$ with the automorphism \eqref{eqn: twisted power series automorphism} produces an irreducible highest weight representation with polynomial highest weight. Thus, it suffices to consider such representations. In this case,
\begin{align*}
    \mu(u) = (1-\gamma_1u^{-1})\cdots (1-\gamma_{2k}u^{-2k}).
\end{align*}
As in \cite{kalinov_yangians_20}, for $n\in \FF_p$, let $[n]$ denote the minimal non-negative representative of $n$ modulo $p$, inducing a total order $0<1<\cdots<p-1$ on $\FF_p$. If $p=0$ let $[n]=n$ for a nonnegative integer $n$. Recall that if $(\alpha, \beta)$ is a restricted dominant integral weight for $GL_2$, then $L(\alpha,\beta,p)$ is a simple $GL_2$-module and remains simple upon restriction to $\gl_2$.
\begin{prop}\label{prop: Y(sp2) tensor product of eval modules is irreducible}
    Suppose that for every $i=1,\ldots, k$ the following condition holds: in characteristic $p > 0$ (resp. $p = 0$), if the multiset $\{\gamma_p+\gamma_q \mid 2i-1\le p<q\le 2k\}$ contains elements of $\FF_p$ (resp. elements of $\ZZ_+$), then $\gamma_{2i-1}+\gamma_{2i}$ is in $\FF_p$ (resp. $\ZZ_+$) is minimal amongst them. Then the representation $V(\mu(u),p)$ of $Y(\xsp_2)$ is isomorphic to the tensor product
    \begin{align}\label{eqn: Y(sp2) tensor product of eval modules}
        L(\gamma_1,-\gamma_2,p)\otimes \cdots\otimes L(\gamma_{2k-1},-\gamma_{2k},p),
    \end{align}
    regarded as a $Y(\xsp_2)$-module obtained by restriction of the $Y(\gl_2)$-module.
\end{prop}
\begin{proof}
    We adapt the proof of Proposition 4.3.2 in \cite{molev_yangians_book_07} following the approach of Proposition 3.2.3 in \cite{kalinov_yangians_20}. By \cref{cor: highest weight of tensor product of eval modules}, if we show the module in question is an irreducible highest weight module, then it has highest weight $\mu(u)$.

    Let $V$ be the module \eqref{eqn: Y(sp2) tensor product of eval modules}, and consider the embedding \eqref{eqn: embed Y(gn) into Y(glN)}. The $(i,j)$-entry of $T'(u)$ is $t_{-j,-i}(u)$, so using \eqref{eqn: embed Y(gn) into Y(glN)}
    \[\s_{-1,1}(u)  = t_{-1,-1}(u)t_{-1,1}(-u) + t_{-1,1}(u)t_{-1,1}(-u).\]
    Since $t_{-1,1}(u)$ acts nilpotentently and $t_{-1,-1}(u)$ acts semisimply (as in \cite[Prop.~3.2.3]{kalinov_yangians_20}), it follows that $\s_{-1,1}(u)$ acts nilpotently. Any submodule $N\subset V$  must then contain a singular vector for $\s_{-1,1}(u)$. Thus, irreducibility reduces to showing $V$ has a unique singular vector.

We use induction to prove that any singular vector is proportional to $\zeta := \zeta_1\otimes \cdots \otimes \zeta_k$, where $\zeta_i$ are the highest weight vectors of the $L(\gamma_{2i-1},-\gamma_{2i},p)$. The case $k=1$ is clear. For $k\ge 2$, write \[\eta = \sum_{r=0}^q (E_{1,-1})^r \zeta_1\otimes \eta_r,\] with $\eta_r \in L(\gamma_3,-\gamma_4,p)\otimes \cdots \otimes L(\gamma_{2k-1},-\gamma_{2k},p)$, $\eta_q\ne 0$, and $q\le \min\{p-1,[\gamma_1+\gamma_2]\}$ (if $p=0$ then $q< \gamma_1+\gamma_2$). Using \eqref{eqn: comultiplication on sij(u)} and the argument of \cite[Prop.~4.3.2]{molev_yangians_book_07}, $\eta_q$ is proportional to $\zeta_2\otimes \cdots \otimes \zeta_k$, so it remains to show $q=0$. 

 If $q\ge 1$, then the argument of \cite[Prop. 4.3.2]{molev_yangians_book_07} gives
    \[q(\gamma_1+\gamma_2-q+1)(\gamma_1+\gamma_3-q+1)\cdots (\gamma_1+\gamma_{2k}-q+1) = 0\]
which is impossible given the assumptions on the $\gamma_i$, since $\gamma_1+\gamma_2$ is minimal among $\gamma_1+\gamma_i$ and  $q\le [\gamma_1+\gamma_2]$. Thus $q=0$, so  $V$ has a unique singular vector.

Finally, this vector generates $V$ by the same argument as in the characteristic zero case, completing the proof.
\end{proof}
\begin{lem}\label{lem: Y(sp2) reordering}
    Any $\gamma_1,\ldots,\gamma_{2k}\in\KK$ can be rearranged to satisfy the conditions of the previous proposition.
\end{lem}
\begin{proof}
    First, pick the $i,j$ such that $[\gamma_i+\gamma_j]$ is minimal out of all the pairs $i,j$ for which this is defined. Then we can set $\gamma_i,\gamma_j$ to be the new $\gamma_1,\gamma_2$. Repeat for the remaining $\gamma_i$.
\end{proof}
Hence, for any $V$ satisfying the conditions of \cref{prop: positive char highest weight}, $V=V(\mu(u),p)$ with \[\mu(u) = (1-\gamma_1u^{-1})\cdots (1-\gamma_{2k}u^{-2k}),\] such that $V(\mu(u),p)$ is isomorphic to \eqref{eqn: Y(sp2) tensor product of eval modules}. Now we are ready to state the classification theorem for $Y(\xsp_2)$.
\begin{thm}\label{thm: Y(sp2) classificiation}
    If the irreducible finite-dimensional representation $V$ of $Y(\xsp_2)$ satisfies the conditions of \cref{prop: positive char highest weight}, then $V=V(\mu(u),p)$ such that there exists monic polynomial $P(u)$ such that $P(u) = P(-u+1)$ and
    \begin{align}\label{eqn: double arrow positive char}
        \frac{\mu(-u)}{\mu(u)} = \frac{P(u+1)}{P(u)}.
    \end{align}
\end{thm}
\begin{proof}
We can simply repeat the first part of the proof of Theorem 4.3.3 in \cite{molev_yangians_book_07}.
\end{proof}

For such $\mu(u)$ as in Theorem \ref{thm: Y(sp2) classificiation} satisfying equation \eqref{eqn: double arrow positive char} for sutiable monic polynomial $P(u)$, we write $\mu(-u) \Longrightarrow \mu(u)$. Now we generalize this to the $Y(\xsp_{2n})$ case. 
\begin{thm}\label{thm: Y(sp2n) classification 1}
    If $V$ satisfies the conditions of \cref{prop: positive char highest weight} then $V=V(\mu(u))$ with
    \begin{align}
        \mu_1(-u)\Longrightarrow \mu_1(u)\longrightarrow \mu_2(u) \longrightarrow \cdots \longrightarrow \mu_n(u)
    \end{align}
\end{thm}
\begin{proof}
    Consider the embedding $Y(\xsp_2) \hookrightarrow Y(\xsp_{2n})$ given by $\s_{ij}(u) \mapsto \s_{ij}(u)$ for $i,j\in\{1,-1\}$. If $V$ satisfies \cref{prop: positive char highest weight}, the induced $Y(\xsp_2)$-module satisfies \cref{prop: sp2 highest weight representation positive char}; its highest weight is $\mu_n(u)$. Since $V$ is finite-dimensional, \cref{thm: Y(sp2) classificiation} gives $\mu_1(-u) \Longrightarrow \mu_1(u)$, and \cref{prop: gln condition holds for gn} finishes the proof.
\end{proof}
We call the polynomials $P_1(u),\ldots,P_n(u)$ corresponding to each of the arrows the \emph{Drinfeld polynomials} of $V(\mu(u))$. 

We also want to go the other way: to construct irreducible representations from Drinfeld polynomials.

\begin{thm}\label{thm: Y(sp2n) classification 2}
    Suppose $p > 2$. Given monic polynomials $P_1(u), \ldots, P_n(u)$ with $P_1(u) = P_1(-u+1)$, there exists a finite-dimensional representation $V(\mu(u),p)$ of $Y(\xsp_{2n})$ such that
    \begin{align*}
        \frac{\mu_{i}(u)}{\mu_{i+1}(u)} = \frac{P_{i+1}(u+1)}{P_{i+1}(u)}, \qquad i=1,\ldots, n-1
    \end{align*}
    and
    \begin{align*}
        \frac{\mu_1(-u)}{\mu_1(u)} = \frac{P_1(u+1)}{P_1(u)}.
    \end{align*}
\end{thm}
\begin{proof}
    We may repeat the argument from the second part of the proof of Theorem 4.3.8 in \cite{molev_yangians_book_07}. The existence of an external grading as in the characteristic zero case ensures that $V(\mu(u),p)$ appears as a subquotient in the cyclic $Y(\xsp_{2n})$-span of the highest vector of $L(\lambda(u))$.
\end{proof} 
Using the notation of the second part of the proof of Theorem 4.3.8 in \cite{molev_yangians_book_07}, the largest value of $-\mu_n$ in a $\xsp_{2n}$-weight $\mu$ appearing in $L(\lambda(u),p)$ coincides with the maximum of $\lambda_{-n} - \lambda_{n}$ for a $\gl_{2n}$-weight $\lambda$ appearing in $L(\lambda(u),p)$.  From the remarks following Theorem 3.2.6 in \cite{kalinov_yangians_20}, this maximum is given by is the sum of the degrees of the Drinfeld polynomials corresponding to $\lambda(u)$, which are $1,\ldots, 1,Q(u),P_2(u),\ldots,P_n(u)$. So the maximum value of $-\mu_n$ in a $\xsp_{2n}$-weight $\mu$ appearing in $L(\lambda(u),p)$ is $\deg Q(u)+\sum_{i=2}^n \deg P_i(u)$. Note that $\deg P_1(u)=2\deg Q(u)$. Therefore, if \[\deg P_1(u) + 2\sum_{i=2}^{n} \deg P_i(u) + 2n < p,\] the corresponding $V(\mu(u))$ satisfies $-2\mu_n+2n<p$ for all $\mu$ appearing in its decomposition into simple $\g_n$-modules. By \cref{lem: irreducible g_n-modules in positive char}, $V$ satisfies the conditions of \cref{prop: positive char highest weight} (and vice versa).

As in the $\gl_n$ case, the Drinfeld polynomials $P_i(u)$ are not unique. If \[\frac{P_i(u+1)}{P_i(u)} = \frac{Q_i(u+1)}{Q(u)},\] then $\frac{P_i}{Q_i} = F_i$, where $F_i(u) = F_i(u+1)$. Thus, $F_i$ is a ratio of products of expressions of the form \[(u+c)(u+1+c)\cdots (u+c+p-1) = (u+c)^p - (u+c)\] for some $c \in \overline{\FF}_p$. Following \cite{kalinov_yangians_20}, we define $q_p(u) = u^p - u$, so $P_i(u)$, for $i \ge 2$, is unique up to multiplication or division by polynomials of the form $q_p(u+c)$. However, if we impose that \[\deg P_1(u) + 2\sum_{i=2}^{n} \deg P_i(u) + 2n < p,\] the polynomials $P_i(u)$ are unique since $\deg q_p(u+c) = p$.          

\begin{cor}
    Every finite-dimensional irreducible representations of $SY(\xsp_{2n})$ that satisfies the conditions of \eqref{prop: positive char highest weight} corresponds to a unique tuple $(P_1(u),\ldots,P_n(u))$ of monic polynomials such that $P_1(u)=P_1(-u+1)$ and $\deg P_1(u) + 2\sum_{i=2}^{n} \deg P_i(u) + 2n < p$.
\end{cor}
\begin{proof}
    This follows from the previous two theorems and \cref{prop: representations of Y(gn) and SY(gn) positive char}.
\end{proof}
\subsubsection{Finite-dimensional irreducible representations of $Y(\o_{2n})$ in positive characteristic} 
In this section, we look at finite-dimensional irreducible representations of $Y(\o_{2n})$. We start by looking at finite-dimensional irreducible representations of $Y(\o_2)$.

\begin{prop}\label{prop: o2 highest weight representation positive char}
    Suppose $V$ is irreducible and finite-dimensional, and as an $SO_2$ and $\o_2$-representation is isomorphic to $V(\mu_1,p)\otimes \cdots \otimes V(\mu_k,p)$ with each $\mu_i$ is a non-positive integer and $p>2\max -\mu_i + 2$. Then $V$ is a highest representation $V=V(\mu(u),p)$.
\end{prop}
\begin{proof}
    The condition implies that $\mu_i$ is dominant integral and in the fundamental alcove (and therefore restricted, so we can freely interchange between the group and the Lie algebra), and hence each $V(\mu_i,p)$ is irreducible and highest weight. Then we can repeat the proof of Theorem 4.2.6 in \cite{molev_yangians_book_07}.
\end{proof}
From now on, assume $V$ satisfies the conditions of \eqref{prop: o2 highest weight representation positive char}.
Write
\begin{align}
    \mu(u) = 1+\mu^{(1)}u^{-1}+\mu^{(2)}u^{-2}+\cdots,\quad \mu^{(r)}\in \overline{\FF}_p.
\end{align}
Let $V$ be an irreducible finite-dimensional representation of $Y(\o_2)$ satisfying the conditions of \cref{prop: positive char highest weight}, so $V=V(\mu(u))$ with
\begin{align*}
    \mu(u) = 1+\mu^{(1)}u^{-1}+\mu^{(2)}u^{-2}+\cdots, \quad \mu^{(r)}\in \overline{\FF}_p.
\end{align*}
\begin{prop}\label{prop: Y(o2) polynomial weight positive char}
There exists an even formal series $g(u)\in 1+u^{-2}\overline{\FF}_p\dbrack{u^{-2}}$ such that $(1+\frac 12u^{-1})g(u)\mu(u)$ is a polynomial in $u^{-1}$.
\end{prop}
\begin{proof}
We may repeat the proof of Proposition 4.4.1 in \cite{molev_yangians_book_07} without any changes.
\end{proof}
As in the characteristic zero case, define $\mu'(u)=(1+\frac 12u^{-1})\mu(u)$. Now we may focus on the representations with $\mu'(u)$ a polynomial in $u^{-1}$. Then write
\begin{align*}
    \mu'(u) =(1-\gamma_1u^{-1})\cdots (1-\gamma_{2k+1}u^{-1})
\end{align*}
for $\gamma_i\in \overline{\FF}_p$. For any $\gamma\in \overline{\FF}_p$, we can make the one-dimensional $\o_2$-representation $V(\gamma)$ into a representation of the twisted Yangian $Y(\o_2)$ via the evaluation homomorphism, and the corresponding highest weight is given by
\[\frac{1+(\gamma+\tfrac{1}{2})u^{-1}}{1+\tfrac{1}{2}u^{-1}},\]
which follows from \eqref{eqn: highest weight of Y(gn)-evaluation module}.
\begin{prop}
Suppose that for every $i=1,\ldots, k$ the following condition holds: In positive characteristic (resp. characteristic zerp), if the multiset $\{\gamma_p+\gamma_q \mid 2i-1\le p<q\le 2k+1\}$ contains elements of $\FF_p$ (resp. $\ZZ_+$), then $\gamma_{2i-1}+\gamma_{2i}$ is in $\FF_p$ (resp. $\ZZ_+$) and is minimal amongst them. Then the representation $V(\mu(u))$ of $Y(\o_2)$ is isomorphic to the tensor product
\begin{align}\label{eqn: Y(o2) irreducible representation tensor product}
L(\gamma_1, -\gamma_2,p) \otimes L(\gamma_3, -\gamma_4,p) \otimes \cdots \otimes L(\gamma_{2k-1}, -\gamma_{2k},p) \otimes V(-\gamma_{2k+1} - 1/2,p).
\end{align}
\end{prop}
\begin{proof}
    The proof is exactly analogous to the proof of \cref{prop: Y(sp2) tensor product of eval modules is irreducible}. The characteristic zero version is Proposition 4.4.2 in \cite{molev_yangians_book_07}.
\end{proof}
As in the symplectic case, we can always reorder the $\gamma_i$ to satisfy the conditions of this proposition.

\begin{thm}\label{thm: Y(o2) classification positive char}
If the irreducible highest weight representation $V$ of $Y(\o_2)$ is finite-dimensional and satisfies the conditions of \cref{prop: Y(o2) polynomial weight positive char} then $V=V(\mu(u),p)$ and there is a pair $(P(u), \gamma)$, where $P(u)$ is a monic polynomial in $u$ with $P(u) = P(-u+1)$ and $\gamma \in \overline{\FF}_p$ with $P(\gamma) \neq 0$ such that
\begin{align}\label{eqn: Y(o2) classification positive char}
\frac{\mu'(-u)}{\mu'(u)} = \frac{P(u+1)}{P(u)} \cdot \frac{u - \gamma}{u + \gamma},
\end{align}
where $\mu'(u) = (1 + \frac 12 u^{-1}) \mu(u)$. 
\end{thm}
\begin{proof}
    We know from above that $V = V(\mu(u),p)$, and it suffices to consider those $\mu(u)$ for which $\mu'(u)=(1+\gamma_1u^{-1})\cdots (1+\gamma_{2k+1}u^{-1})$ is a polynomial. As a $\o_2$-module, we know $V$ has the form
    \[L(\gamma_1,-\gamma_2)\otimes \cdots \otimes L(\gamma_{2k-1},-\gamma_{2k})\otimes V(\gamma_{2k+1}-1/2). \]
    Since this is finite-dimensional, we get that $\gamma_{2i-1}+\gamma_i\in \FF_p$ as in \cite{kalinov_yangians_20}. Then set
    \begin{align*}
        \lambda_1(u) &= (1+\gamma_1u^{-1})(1+\gamma_3u^{-1})\cdots (1+\gamma_{2k-1}u^{-1})\\
        \lambda_2(u) &= (1-\gamma_2u^{-1})(1-\gamma_4u^{-1})\cdots (1-\gamma_{2k}u^{-1}).
    \end{align*}
    As in the symplectic case, we get a monic polynomial $Q(u)$ in $u$ such that
    \begin{align*}
        \frac{\lambda_1(u)}{\lambda_2(u)}=\frac{Q(u+1)}{Q(u)}.
    \end{align*}
    Then set \[P(u)=Q(u)Q(-u+1)(-1)^{\deg Q},\] and $\gamma=-\gamma_{2k+1}$; it is easily checked this does the job. The $P(\gamma)\ne 0$ condition follows from the ordering of the $\gamma_i$ in the same way as in the characteristic zero case.
\end{proof}

Now we look at the general case of finite-dimensional irreducible representations of $Y(\o_{2n})$, where $n\ge 2$ is a positive integer. We know that the irreducible highest weight representations $V(\mu(u))$ of $Y(\o_{2n})$ are parametrized by $n$-tuples 
\begin{align}
    \mu(u) = (\mu_1(u),\ldots,\mu_n(u)).
\end{align}
We may regard $Y(\o_2)$ as a subalgebra of $Y(\o_{2n})$ generated by $\s_{ij}(u)$, $i,j\in \{n,n+1\}$. Let $\xi$ be the highest weight vector of the $Y(\o_{2n})$-module $V(\mu(u),p)$.
\begin{lem}\label{lem: restrict Y(o2n) module to Y(o2)}
    The $Y(\o_2)$-module $V=Y(\o_2)\xi$ is irreducible and isomorphic to the highest weight module $V(\mu_n(u),p)$.
\end{lem}
\begin{proof}
    We can use the same proof as Lemma 4.11 of \cite{molev_yangians_book_07}.
\end{proof}
Consider the mapping $i\mapsto i'$ for $i\in I_N$ that interchanges indices $1$ and $-1$ while leaving all other indices unchanged. Using the defining relations, it is easily checked that the mapping
\begin{align}\label{eqn: sharp automorphism}
    \s_{ij}(u)\mapsto \s_{i'j'}(u)
\end{align}
is an automorphism of $Y(\o_{2n})$. 

We now briefly return to the $Y(\o_2)$ case. For the finite-dimensional representation $V(\mu(u))$ of $Y(\o_2)$, we know that $V(\mu(u))$ is associated with a unique pair $(P(u),\gamma)$. Let $V(\mu(u))^\sharp$ be the composition of the $Y(\o_2)$-action on $V(\mu(u))$ with the automorphism \eqref{eqn: sharp automorphism}.
\begin{lem}\label{lem: sharp highest weight positive char}
    The $Y(\o_2)$-module $V(\mu(u),p)^\sharp$ is isomorphic to $V(\mu^\sharp(u),p)$, where the series $\mu^\sharp(u)$ is given by
    \begin{align*}
        \mu^\sharp(u) = \mu(u) \cdot \frac{u-\gamma+1}{u+\gamma}.
    \end{align*}
    In particular, the $Y(\o_2)$-module $V(\mu(u),p)^\sharp$ is associated with the pair $(P(u),-\gamma+1)$. 
\end{lem}    
\begin{proof}
    We use the same proof as Lemma 4.4.13 in \cite{molev_yangians_book_07}.
\end{proof}
Now, we can state the main classification theorems. Given an \textit{admissible} series $\mu(u)$, we define $\mu^\sharp(u)$ as above.
\begin{thm}
    If $V$ is a finite-dimensional irreducible $Y(\o_{2n})$-module, $n\ge 2$, and satisfies the conditions of \cref{prop: positive char highest weight}, then $V=V(\mu(u),p)$, the series $\mu_{1}(u)$ is admissible, and one of the following relations holds:
\begin{align}
&\mu_1(-u)\Longrightarrow\mu_{1}(u) \longrightarrow \mu_{2}(u) \longrightarrow \cdots \longrightarrow \mu_{n}(u)\\
&\mu_1^\sharp(-u)\Longrightarrow\mu_{1}^\sharp(u) \longrightarrow \mu_{2}(u) \longrightarrow \cdots \longrightarrow \mu_{n}(u).
\end{align}
Moreover, in positive characteristic, the first one will always hold.
\end{thm}
\begin{proof}
    Consider the module $V(\mu_1(u),p)$ of \cref{lem: restrict Y(o2n) module to Y(o2)}. Examining the $\o_2$-weights and using that $V(\mu(u),p)$ satisfies the conditions of \cref{prop: positive char highest weight}, we see that this module meets the conditions of \cref{prop: o2 highest weight representation positive char}. Suppose $V(\mu_n(u))$ is associated with the number $\gamma\in\overline{\FF}_p$. Since $V$ satisfies the conditions of \cref{prop: positive char highest weight}, it follows that the $\o_2$-weights of $V(\mu_n(u),p)$ must all be in $\FF_p$. Now $V(\mu_1(u))$ must be given by \eqref{eqn: Y(o2) irreducible representation tensor product}. Since the $L(\gamma_{2i-1},-\gamma_{2i})$ must be finite dimensional, $\gamma_{2i-1}+\gamma_{2i}\in \FF_p$. It follows that the $\o_{2}$-weights of $L(\gamma_{2i-1},-\gamma_{2i})$ are all in $\FF_p$. Thus, we must have that $-\gamma_{2k+1}-1/2$ is in $\FF_p$, which is equivalent to $\gamma$ being $1/2$ plus an integer.
    Proceeding as in the characteristic zero case, it follows that the module $V(\mu(u))^\sharp$ has highest weight
    \begin{align*}
        \mu^{\sharp}(u) = (\mu_1(u),\mu_2(u),\ldots,\mu_n^\sharp(u))
    \end{align*}
    cf.~the beginning of the proof of Theorem 4.4.14 in \cite{molev_yangians_book_07}. Moreover, composing the automorphism \eqref{eqn: sharp automorphism} (restricted to $Y(\o_{2n})$) with the action of $\o_{2n}$ on $V(\mu(u),p)$ yields another irreducible highest weight representation, so $V(\mu(u))^\sharp$ also satisfies the conditions of \cref{prop: positive char highest weight}.

    Using \eqref{eqn: Y(o2) classification positive char}, we obtain
    \begin{align*}
        \frac{\mu_1(-u)}{\mu_1(u)} = \frac{P(u+1)}{P(u)}\cdot \frac{u+\tfrac 12}{u+\gamma}\cdot \frac{u-\gamma}{u-\tfrac 12} = \frac{P(u+1)}{P(u)}\cdot \frac{u+\tfrac 12}{u-\delta+\tfrac{1}{2}}\cdot \frac{u+\delta - \tfrac{1}{2}}{u-\tfrac 12}.
    \end{align*}
    where $\delta = -\gamma+\tfrac 12$ and $P(u)=P(-u+1)$. Note that $\delta$ is an integer from our previous discussion. In positive characteristic, or in characteristic zero, and $\gamma\in 1/2-\ZZ_+$, $\delta$ is a non-negative integer, so we can set \[Q(u) = P(u)(u-\delta+\tfrac 12)\cdots (u-\tfrac 12) (u-\tfrac 12)\cdots (u+(\delta-1)-\tfrac 12).\]
    Since we have multiplied by an even number of linear factors, we have $Q(u) = Q(-u+1)$. We also have
    \[\frac{\mu_1(-u)}{\mu_1(u)} = \frac{Q(u+1)}{Q(u)},\]
    so the first condition holds. If $\gamma\in 1/2+\ZZ_+$ then $\delta$ is a non-positive integer. Consider the $\delta^\sharp,\gamma^\sharp$ associated with $V^\sharp$. By \cref{lem: sharp highest weight positive char}, $\gamma^\sharp=-\gamma+1$, so $\delta^\sharp  = -\gamma^\sharp +1/2 = -(-\gamma+1)+1/2 = \gamma-1/2 = -\delta$. Hence $V^\sharp$ satisfies the first condition, so $V$ satisfies the second.
\end{proof}

As before, we want to go the other way as well.
\begin{thm}\label{thm: Y(o2n) classification 2}
    Suppose $p > 2$. Given monic polynomials $P_1(u), \ldots, P_n(u)$ with $P_1(u) = P_1(-u+1)$, there exists a finite-dimensional representation $V(\mu(u),p)$ of $Y(\o_{2n})$ such that
    \begin{align*}
        \frac{\mu_{i}(u)}{\mu_{i+1}(u)} = \frac{P_{i+1}(u+1)}{P_{i+1}(u)}, \qquad i=1,\ldots, n-1
    \end{align*}
    and
    \begin{align*}
        \frac{\mu_1(-u)}{\mu_1(u)} = \frac{P_1(u+1)}{P_1(u)}.
    \end{align*}
\end{thm}
\begin{proof}
    Completely analogous to the proof of \cref{thm: Y(sp2n) classification 2}.
\end{proof}
In characteristic zero, it follows that the irreducible representations of $Y(\o_{2n})$ are indexed by Drinfeld polynomials $P_1(u),\ldots, P_n(u)$ with $P_1(u)=P_1(-u+1)$ and a choice of the parameter $\gamma$ in one of the three subsets
\begin{align}\label{eqn: subsets}
    \{1/2\},-1/2-\ZZ_+,3/2+\ZZ_+,
\end{align}
where the first subset corresponds to both conditions (since $\delta=0=-\delta=\delta^\sharp$), the second corresponds the first condition, and the third corresponds to the second condition. 

In positive characteristic, similarly to the symplectic case, we can get that $V(\mu(u))$ satisifies the conditions of \cref{prop: positive char highest weight} if and only if $\deg P_1(u) + 2\sum_{i=2}^{n} \deg P_i(u) + 2n < p$.
\begin{cor}
    Finite-dimensional irreducible representations $V$ of $SY(\o_{2n})$, $n\ge 2$, satisfying the conditions of \cref{prop: positive char highest weight} are highest weight modules $V=V(\mu(u))$ and are parametrized by tuples $(P_1(u),\ldots, P_n(u))$ of monic polynomials with $P_1(u)=P_1(-u+1)$ and in positive characteristic, $\deg P_1(u) + 2\sum_{i=2}^{n} \deg P_i(u) + 2n < p$, and in characteristic zero, a choice of one of the subsets in \eqref{eqn: subsets}.
\end{cor}

\section{Deligne's categories for classical groups}\label{sec: deligne categories}
We will assume in this section that the reader has at least a basic understanding of symmetric tensor categories (STCs) and in particular has some familiarity with the Deligne category $\Rep GL_t$. A comprehensive reference is \cite{etingof2016tensor}, and a quick expository reference is \cite{etingof2021lectures}, whose notations and conventions we will mainly follow. Our primary reference for this section will be \cite{etingof_complex_rank_2_16}. We will work over a fixed, algebraically closed field $\mathbb{K}$ of characteristic zero, and we will use the bold-face notation $\bfRep G$ to denote the ordinary representation category of finite-dimensional representations for a subgroup $G$ of the general linear group $GL_N$ for a positive integer $N$. 

\subsection{The Deligne categories \texorpdfstring{$\Rep(O_t)$}{RepOt} and \texorpdfstring{$\Rep(Sp_t)$}{RepSpt}}
Let $V = \mathbb{K}^n$ denote the natural representation of $GL_N$. As an $O_N$-module, $V$ is faithful and self-dual and therefore generates the category $\bfRep O_n$, so every simple $O_N$-module appears in some $V^{\otimes r}$ for a sufficiently large $r$. By tensor-hom adjunction and self-duality, we see that
\[\Hom_{O_N}(V^{\otimes r}, V^{\otimes s}) = \Hom_{O_N}(V^{ \otimes r+s}, \un),\]
where $\un$ is the trivial representation. By the First Fundamental Theorem of invariants for $O_N$, multilinear $O_N$-invariant forms on $V^{\otimes r+s}$ are given by the algebra generated by contractions (see \cite{DELIGNE20184} for more details), which means that $r+s$ must be even. By semisimplicity of $\bfRep O_N$, one can conclude that $\bfRep O_N$ is the Karoubian closure of the subcategory given by objects of the form $[r] \coloneqq V^{\otimes r}$ and hom spaces given as above.
\par
It can be show via Brauer duality that a spanning set for $\Hom_{O_N}([r], [s])$ are the perfect matchings on a set of $(r+s)$ dots, which can be visualized as diagrams with a row of $r$ dots below and a row of $s$ dots above, with an arc connecting the two dots which are perfectly matched. Composition is given by first stacking diagrams vertically, then concatenating arcs in the natural way, and finally erasing loops; for each loop we multiply the diagram by $n$.
\par
As an example, here is the composition of a diagram $A$ in $\Hom_{O_N}([3], [5])$ and a diagram $B$ in $\Hom_{O_N}([5],[3])$ to give a diagram $A \circ B \in \Hom_{O_N}([5], [5])$:
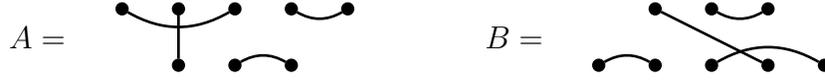
\[\begin{tikzpicture}[line width=1pt, scale=.75]
	\foreach\x in {1,...,5}{
		\node[bV] (t\x) at (\x,1){};}
	\foreach\x in {1,...,3}{
		\node[bV] (b\x) at (\x+1,0){};}
        \node (a) at (-0.5,0.5) {$A=$};
	\draw [bend left] (t5) to (t4) (b2) to (b3) (t3) to (t1);
        \draw (t2)--(b1);
\end{tikzpicture}\hspace{16mm}\begin{tikzpicture}[line width=1pt, scale=.75]
	\foreach\x in {1,...,3}{
		\node[bV] (t\x) at (\x+1,1){};}
	\foreach\x in {1,...,5}{
		\node[bV] (b\x) at (\x,0){};}
        \node (b) at (-0.5,0.5) {$B=$};
	\draw [bend left] (t3) to (t2)  (b3) to (b5) (b1) to (b2);
        \draw (t1)--(b4);
	
\end{tikzpicture}\]
with composition given by
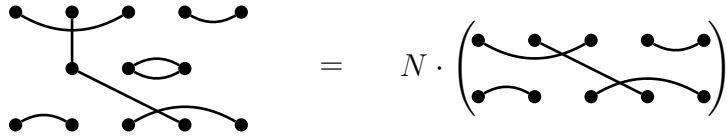
\[
\begin{tikzpicture}[line width=1pt, scale=.75]
	\foreach\x in {1,...,5}{
		\node[bV] (t\x) at (\x-2,1){};}
	\foreach\x in {1,...,3}{
		\node[bV] (b\x) at (\x-1,0){};}
        \foreach\x in {1,...,5}{
		\node[bV] (c\x) at (\x-2,-1){};}
	\draw [bend left] (t5) to (t4) (b2) to (b3) (t3) to (t1);
        \draw (t2)--(b1);
        \draw [bend left] (b3) to (b2)  (c3) to (c5) (c1) to (c2);
        \draw (b1)--(c4);
        \node (=) at (4.6,0) {$=$};
        \node (ab) at (8.76,0) {$N\cdot \left(\phantom{a...................\displaystyle\sum_{i=1}}\right)$};
        \foreach\x in {1,...,5}{
		\node[bV] (t1\x) at (\x+6.2,0.5){};}
        \foreach\x in {1,...,5}{
		\node[bV] (c1\x) at (\x+6.2,-0.5){};}
	\draw [bend left] (t15) to (t14)  (t13) to (t11);
        \draw (t12)--(c14);
        \draw [bend left]  (c13) to (c15) (c11) to (c12);
	
\end{tikzpicture}
\]

For arbitrary $t\in \CC$, we can define the category $\widetilde{\Rep}(O_t)$ with objects $[r]$, with $r\in \ZZ_+$. The space of morphisms $\Hom([r_1],[r_2])$ is spanned by the Brauer diagrams as above, with the same composition law, where if we remove a loop, we multiply by $t$. The category $\Rep(O_t)$ is defined to be the Karoubian closure of $\widetilde{\Rep}(O_t)$. We define $\widetilde{\Rep}(Sp_t)$ similarly. In both cases, we let $V$ be the fundamental object $[1]$. For $t\not\in \ZZ$, the categories $\Rep(O_t)$ and $\Rep(Sp_t)$ are semisimple. We will not consider the case of $t\in \ZZ$ in this paper.\footnote{In the case $t \not\in \Z$, semisimplicity implies that these categories are abelian and therefore are symmetric tensor categories. On the other hand, when $t \in \Z$, they are not abelian, and it is much more difficult to show the abelian envelope exists (see \cite{ehs_envelope, Coulembier_2021}).}

Then $V$ is equipped with a symmetric or alternating isomorphism $\psi: V\to V^*$. If $\psi$ is symmetric, the group $O(V)$ is cut out inside $V\otimes V^*$ by the equations $AA' = A'A = \Id$, where $A'$ is the adjoint of $A$ with respect to $\psi$. The map $\theta\in \End(V\otimes V^*)$ that sends $A$ to $A'$ is given by
\begin{align}\label{eqn: adjoint wrt form}
    V\otimes V^* \xrightarrow{1\otimes \psi^{-1}} V\otimes V \xrightarrow{\ \sigma\ } V\otimes V \xrightarrow{1\otimes \psi} V\otimes V^*,
\end{align}
where $\sigma$ exchanges factors in the tensor product. This can be thought of as taking the adjoint with respect to the symmetric bilinear form on $V$ that is induced by $\psi$, analogously to \eqref{eqn: transpose wrt form}. In the integral rank case, if $G$ is the matrix representation of $\psi$, then $\theta$ sends $A$ to $G^{-1}A^tG$. Our choice of $\theta$ is the negative of that in \cite{etingof_complex_rank_2_16}. If $\psi$ is alternating, the group $Sp(V)$ is defined in the same way. 

This definition works if $V$ is replaced with any object that has a symmetric/alternating form. So when $V$ is the fundamental object of $\Rep(G)$, we write $O_t = O(V)$ or $Sp_{t}=Sp(V)$ corresponding to the orthogonal and symplectic cases, respectively.

We next define the Lie algebras $\o_t$ and $\xsp_t$. First, $\gl(V) = V\otimes V^*$ is naturally an associative algebra, and thus also a Lie algebra, using the commutator. In the orthogonal (resp. symplectic) case, $\o(V)$ (resp. $\xsp(V)$) is defined as $\ker(-\theta-\Id)$. When $V$ is the fundamental object, we get \begin{align*}
    \o_t=\o(V)=\mathrm{Lie}O(V)=\mathrm{Lie}O_t & \hspace{10mm} \text{in the orthogonal case} \\
    \xsp_{t}=\xsp(V)=\mathrm{Lie}Sp(V)=\mathrm{Lie}Sp_{t} & \hspace{10mm} \text{in the symplectic case}.
\end{align*}

Set $\g=\o_t$ or $\xsp_{t}$, corresponding to $G=O_t$ or $Sp_{t}$. We define the universal enveloping algebra to be a quotient of the tensor algebra of $\g$:
\[U(\g)\coloneqq T(\g)/(f(\g\otimes \g)),\]
where $f=i_2-i_2\sigma-i_1c$, and $i_k$ is the inclusion of the $k$-th graded component, and $c$ is the commutator. Note that $\theta\circ\theta=\Id$ implies $(-\theta-\Id)\circ(1-\theta) = 0$, so the image of $V\otimes V^*$ under $1-\theta$ lies in $\g$. Consequently, we have the map \[V\otimes V^*\xrightarrow{1-\theta}\g \hookrightarrow U(\g).\] 

Thus, for an algebra, $A$, with generators given by a set of maps $A_i:V\otimes V^*\to A$, an algebra homomorphism $A\to U(\g)$ is specified by giving a function sending each $A_i$ to an element of $\Hom(V\otimes V^*,U(\g))$.

\subsection{Ultraproduct construction}
We will employ an alternative presentation of $\Rep(G)$ using ultraproducts. Throughout, let $\mcF$ be a fixed non-principal ultrafilter. For background on ultrafilters and ultraproducts, see \cite[Section 1.2]{kalinov_yangians_20}.

Consider the ultraproduct of countably infinite copies of $\ov{\QQ}$: $\prod_{\mcF} \ov{\QQ}$. By \L o\'s's theorem, $\prod_{\mcF} \ov{\QQ}$ is a field. Moreover, it has characteristic $0$ since for any $k\in \ZZ_{\ne 0}$, $k = \prod_{\mcF}k$. Additionally, it has cardinality of the continuum. By Steinitz's Theorem, it follows that 
\begin{align}\label{eqn: ultraproduct-isomorphism}
    \prod_{\mcF}\ov{\QQ}\cong \CC.
\end{align}
We start with the case of $t$ being a transcendental number. By composing the (non-canonical) isomorphism \eqref{eqn: ultraproduct-isomorphism} with a suitable automorphism of $\CC$ (considered as a vector space over $\overline{\QQ}$), we may assume that the isomorphism \eqref{eqn: ultraproduct-isomorphism} maps $\prod_{\mcF} 2n$ to $t$. See \cite{kalinov_yangians_20} for details.

Set $t_n=n$, so that $\prod_{\mcF}2t_n=t$ and set $G_n = O_{t_n}$ or $Sp_{t_n}$ correspondingly.
Next, let $\bfRep(G_n) = \bfRep(G_n,\ov{\QQ})$, and set $\widehat{\mcC} =\prod_{\mcF} \bfRep(G_n)$. Let $V_{t_n}$ denote the fundamental representation of $G_n$, and let $V=\prod_{\mcF}V_{t_n}$. We have the following theorem (cf. \cite[Theorem 2.11 (ii)]{etingof_complex_rank_2_16}).
\begin{thm}\label{thm: universal propety of Rep(O_t)}
    If $\mcC$ is a rigid tensor category, then isomorphism classes of (possibly non-faithful) symmetric tensor functors $\Rep(G)\to\mcC$ are in bijection with isomorphism classes of objects $X\in\CC$ of dimension $t$ with a symmetric/alternating (in the orthogonal/symplectic cases) isomorphism $X\to X^*$. The bijection is given by $F = F(V_t)$.
\end{thm}
Here $V_t$ is the fundamental object of $\Rep(G)$.
We now give the ultraproduct construction of $\Rep(G)$. 
\begin{thm}\label{thm: ultraproduct construction}
    The full subcategory of the $\prod_{\mcF}\ov{\QQ}$-linear category $\widehat{\mcC}$ generated by $V$ under tensor products, direct sums, and direct summands is equivalent to the $\CC$-linear category $\Rep(G)$ under the isomorphism $\prod_{\mcF}\ov{\QQ}\cong \CC$ with $\prod_{\mcF}2t_n=t$.
\end{thm}
\begin{proof}
    We will use $V_t$ to refer to the tautological object of $\Rep(G)$, and $V=\prod_{\mcF}V_{t_n}$ to denote the generator of the full subcategory of $\widehat{\mcC}$ in the statement of the theorem.

    Let $\mcC$ be the full subcategory of $\widehat{\mcC}$ generated by $V$ under tensor products, direct sums, and direct summands. The object $V=\prod_\mcF V_{t_n}$ has dimension $t$ as the dimension of $V_{t_n}$ is $2t_n$ and $\prod_\mcF 2t_n = t$. Additionally, since $V_{t_n}^* = V_{t_n}$, by \L o\'s's thoerem we have  $V = V^*$. 
    Using \cref{thm: universal propety of Rep(O_t)}, we obtain a tensor functor $F: \Rep(G)\to \widehat{\mcC}$ with $F(V_t) =V$. Since $\Rep(G)$ is generated by $V_t$ under tensor products, direct sums, and direct summands, it follows that the image of $\Rep(G)$ is the full subcategory $\mcC$ of $\widehat{\mcC}$. So we know that $F: \Rep(G)\to \widehat{\mcC}$ is essentially surjective, and we need to show that $F$ is fully faithful. To do this, we show
    \begin{align}\label{homsets of Rep(O_t) and ultraproduct}
        \Hom_{\Rep(G)}(V_t^{\otimes r_1},V_t^{\otimes r_2}) = \Hom_{\mcC}(V^{\otimes r_1}, V^{\otimes r_2}) 
    \end{align}
    and that the composition maps are the same.
    This implies the result because $\Rep(G)$ and $\mcC$ can be obtained as the Karoubian envelope of the additive envelope of $V_t^{\otimes r}$ and $V^{\otimes r}$, respectively. 
    Note that by \L o\'s's theorem we have
    \begin{align*}
        \Hom_{\mcC}(V^{\otimes r_1}, V^{\otimes r_2}) = \prod_{\mcF}\Hom_{G_n}(V_{t_n}^{\otimes r_1},V_{t_n}^{\otimes r_2})
    \end{align*}
    Now if $r_1+r_2$ is odd, then both sides of \cref{homsets of Rep(O_t) and ultraproduct} are empty, as $\Hom_{G_{t_n}}(V_{t_n}^{\otimes r_1},V_{t_n}^{\otimes r_2})$ is empty. Otherwise, $r_1+r_2=2m$. For sufficiently large $n$, $\Hom_{G_n}(V_{t_n}^{\otimes r_1},V_{t_n}^{\otimes r_2})$ has a basis which can be represented by the Brauer diagrams matching $r_1$ dots in the top row to $r_2$ dots in the bottom row. In particular, there are $\frac{(2m)!}{m!2^m}$ such diagrams, and $\Hom_{G_n}(V_{t_n}^{\otimes r_1},V_{t_n}^{\otimes r_2})$ has dimension $\frac{(2m)!}{m!2^m}$, for sufficiently large $n$. Thus $\prod_{\mcF}\Hom_{G_n}(V_{t_n}^{\otimes r_1},V_{t_n}^{\otimes r_2})$ has a basis consisting of these Brauer diagrams. Since by definition, they are a basis of $\Hom_{\Rep(G)}(V_t^{\otimes r_1},V_t^{\otimes r_2})$, \eqref{homsets of Rep(O_t) and ultraproduct} follows.

    Now we show that the composition maps are the same. In both categories, composition is given by vertical concatenation of diagrams. So we only need to look at what happens when we delete a loop. Each loop deletion can be simplified to $\coev\circ \ev$, which corresponds to multiplication by $2t_n$ in the $G_n$ case. So in $\mcC$, we multiply by $\prod_{\mcF}2t_n$, which is just $t$, and this matches the rule in $\Rep(G)$. This completes the proof.
\end{proof}
Next, consider the case where $t$ is algebraic. It follows a similar story as the transcendental case, except instead of taking the ultraproduct of $\bfRep(G)$ over copies of $\overline{\QQ}$, we take them over copies of $\overline{\FF}_{p_n}$, with $p_n$ being an increasing sequence of primes. As shown in \cite{kalinov_yangians_20}, we can choose an increasing sequence $p_1,p_2,\ldots$ of primes, and an increasing sequence $t_1,t_2,\ldots$ of positive integers such that: $\prod_{\mcF}\overline{\FF}_{p_n} \cong \CC$, and under this isomorphism, $\prod_{\mcF} 2t_n=t$. Moreover, the minimal polynomial $q(x)$ of $t$ satisfies $q(2t_i)=0$ in $\overline{\FF}_{p_i}$. Then, we can repeat the argument of the above proof to get the following theorem.
\begin{thm}\label{thm: ultraproduct construction 2}
    The full subcategory of the $\prod_{\mcF}\overline{\FF}_{p_n}$-linear category $\widehat{\mcC}$ generated by $V$ under tensor products, direct sums, and direct summands is equivalent to the $\CC$-linear category $\Rep(G)$ under the isomorphism $\prod_{\mcF}\ov{\FF}_{p_n}\cong \CC$ with $\prod_{\mcF}2t_n=t$.
\end{thm}

We now state the classification of irreducible objects in $\Rep(G)$, see \cite{utiralova_21} and \cite{etingof_complex_rank_2_16}.

\begin{thm}\label{thm: simple objects of Rep(G)}
    The simple objects of $\Rep(G)$ are labeled by tuples $\lambda = (\lambda_1,\ldots, \lambda_k)$ of negative integers with $\lambda_1\ge \cdots \ge \lambda_k$. Moreover, if $V_\lambda$ is the simple object corresponding to some negative partition $\lambda$, we have
    \[V_{\lambda}=\prod_{\mcF} V_{\lambda}^{(n)},\]
    where for sufficiently large $n$ we define $V_{\lambda}^{(n)}$ to be the simple $G_n$-module with the highest weight $[\lambda]_{t_n}$ given by
    \[[\lambda]_{t_n}=\sum_{i=n-k+1}^{n}\lambda_{n+1-i}\eps_i\in E_n.\]
    Note that for sufficiently large $n$ the weight $\lambda$ is indeed integral and dominant.
\end{thm}

\section{Twisted Yangians in complex rank}\label{sec: yangians in complex rank}
\subsection{The twisted Yangian \texorpdfstring{$Y(\g)$}{Y(g)}}
Let $\g = \o_t$ or $\xsp_t$, and let $G = O_t$ or $Sp_t$, respectively. There is a $G$-equivariant isomorphism $\psi: V \to V^*$, symmetric in the orthogonal case and alternating in the symplectic case. Set 
\[(\pm,\mp)=\begin{cases}
    (+,-) & \text{if } \g=\o_t,\\ 
    (-,+) & \text{if } \g=\xsp_{t}.
\end{cases}\]
For each $i \in \mathbb{Z}_{>0}$, let $V_i \cong V$ be a copy of $V$ in $\Rep(G)$. Consider the tensor algebra
\[B = T\left(\bigoplus_{i=1}^\infty V_i\otimes V_i^*\right),\]
which is an ind-object of $\Rep(G)$. Note that $B$ is generated by images of the maps \[\mcS_i: V \otimes V^* \to V_i \otimes V_i^* \subset B.\] We introduce the formal power series
\[\mcS(u) = 1+\sum_{i>0}\mcS_iu^{-i} \in \Hom(V\otimes V^*,B)\dbrack{u^{-1}},\]
where $1$ is regarded as an element of $\Hom(V \otimes V^*, B)$ via the inclusion $\bbone \hookrightarrow B$.
For convenience, set $\mcS_0 = 1$. Via tensor-hom adjunction, we may also view
\[\mcS(u) \in \Hom( V,V\otimes B)\dbrack{u^{-1}}.\]
Define the $R$-matrix
\[R(u) = 1-\frac{\sigma}{u} \in \Hom(V\otimes V, V\otimes V)\dbrack{u^{-1}},\]
where $\sigma$ is the flip map on $V \otimes V$.

Define 
\[\mcS'(u) = \sum_{i\ge 0} \mcS_i'u^{-i},\qquad \mcS_i'=\mcS_i\circ \theta,\]
so that $\mcS'(u) \in \Hom(V \otimes V^, B)\dbrack{u^{-1}}$. By tensor-hom adjunction, $\sigma$ may be regarded as an element of $\Hom(V \otimes V^* \otimes V \otimes V^*, \bbone)$. Define
\[\sigma'= \sigma\circ (1_{V\otimes V^*}\otimes \theta) \in \Hom(V\otimes V^*\otimes V\otimes V^*, \bbone),\]
and set
\[R'(u) =1 -\frac{\sigma'}{u}.\]
Let $V_I \cong V_{II} \cong V$ be two copies of $V$. For $i \in {I, II}$, let
\[\mcS^{i}(u)\in \Hom(V_I\otimes V_{II}, V_I\otimes V_{II}\otimes B)\dbrack{u^{-1}}\]
be the series $\mcS(u)$ acting on the $V_i$ component. Set
\begin{align*}
    Q(u,v) = (u-v)&(-u-v)\Big[ R(u-v)\mcS^I(u)R'(-u-v)\mcS^{II}(v) \\
&\ -\mcS^{II}(v)R'(-u-v)\mcS^I(u)R(u-v)\Big]
\end{align*}
as an element of $\Hom(V_I \otimes V_{II}, V_I \otimes V_{II} \otimes B)\dbrack{u^{-1}, v^{-1}}$. Via tensor-hom adjunction, we may also regard $Q(u,v)$ as an element of $\Hom(V_I \otimes V_I^* \otimes V_{II} \otimes V_{II}^*, B)\dbrack{u^{-1}, v^{-1}}$.

Finally, define
\begin{align*}
    K(u)   = \mcS'(-u) \mp \mcS(u) - \frac{\mcS(u)-\mcS(-u)}{2u} \in \Hom(V,V\otimes B)\dbrack{u^{-1}},
\end{align*}
which, again by adjunction, may be regarded as an element of $\Hom(V \otimes V^*, B)\dbrack{u^{-1}}$. 

Write
\begin{align*}
    Q(u,v) =\sum_{i,j} Q_{i,j}u^{-i}v^{-j} \quad \text{and}\quad K(u) =\sum_{i} K_{i}u^{-i}.
\end{align*}
Then, the definition of the Yangian is as follows.
\begin{defn}
    The twisted Yangian of $\g$, denoted $Y(\g)$, is the algebra which is the quotient of $B$ by the quadratic relations given by the $Q_{ij}$ and the $K_{ij}$.
\end{defn}
We now will use ultraproducts and \L o\'s's theorem to generalize \cref{prop: eval homomorphism and embedding}. Since our definition reduces to \cref{subsec: twisted yangians in integer rank} in integer rank, we can use \L o\'s's theorem to get \[Y(\g) = \prod_{\mcF} Y(\g_n).\]

Moreover, $\theta$ and $\theta_{t_n}$ satisfy the same relations in their corresponding categories, so again, by \L o\'s's theorem, we get \[\theta=\prod_{\mcF} \theta_{t_n}.\] 
\begin{prop}\label{prop: evaluation and embedding in complex rank}
    We have the embedding $i: U(\g)\hookrightarrow Y(\g)$ given by
    \begin{align}\label{eqn: embedding in complex rank}
        (1-\theta) \mapsto \mcS_1.
    \end{align}
    Here $(1-\theta)$ is a map $V\otimes V^*\to U(\g)$, which generates $U(\g)$, and $\mcS_1$ is a map $V\otimes V^* \to Y(\g)$.
    Furthermore, we have the evaluation map $\varrho\colon Y(\g)\to U(\g)$ given by
    \begin{align}\label{eqn: eval in complex rank}
        \mcS(u) \mapsto 1 + (1 - \theta)\cdot\left(u+\frac{1}{2}\right)^{-1}.
    \end{align}
    Here the coefficients of $\mcS(u)$ are maps from $V\otimes V^*\to Y(\g)$, which generate $Y(\g)$, and the coefficients of the right hand side are maps from $V\otimes V^*$ to $U(\g)$.
\end{prop}
\begin{proof}

    First of all, consider the map \eqref{eqn: embedding in complex rank}.
    In integer rank, this map factors through $U(\g_n)$. Since factoring through is a first order property, we can conclude from \L o\'s's theorem that this map factors through for the ultraproduct $Y(\g)$ as well. 

    Next, consider the map \eqref{eqn: eval in complex rank}, which immediately gives an algebra homomorphism from $B$ to $U(\g_t)$. This map factors through in integer rank when we pass from $B$ to $Y(\g_t)$, and since factoring through by the $Q_{ij}$ and the $K_i$ is a first-order property, it follows from \L o\'s's theorem that we have a map from $Y(\g) \to U(\g)$ that extends the original map.
\end{proof}
\begin{prop}
    If $g(u)=1+g_1u^{-2}+g_2u^{-4}+\cdots\in \CC\dbrack{u^{-2}}$ is an even power series, then
    \begin{align}\label{eqn: twisted power series automorphism complex rank}
        \mcS(u)\mapsto g(u)\mcS(u)    
    \end{align}
    defines an automorphism of $Y(\g)$.
\end{prop}
\begin{proof}
    Replacing $\mcS(u)$ with $g(u)\mcS(u)$ in the defining relations $Q(u,v)$ and $K(u,v)$ just multiplies by some power of $g(u)$, so the defining relations continue to hold.
\end{proof}
The twisted Yangian $Y(\g)$ embeds naturally into the Yangian $Y(\gl_t)$. The Yangian $Y(\gl_t)$ can be defined within both $\Rep(O_t)$ and $\Rep(Sp_t)$ exactly as in \cite{kalinov_yangians_20}. Similarly, we define the series $T'(u)$ analogously to $\mcS'(u)$. We record the following useful fact.
\begin{prop}
    We have the embedding $Y(\g)\hookrightarrow Y(\gl_t)$ given by
    \begin{align}\label{eqn: embedding Y(g) into Y(gl_t)}
        \mcS(u) \mapsto T(u)T'(-u).
    \end{align}
\end{prop}
\begin{proof}
    Consider the map in the integral rank case, which is just \eqref{eqn: embed Y(gn) into Y(glN)}. This defines an embedding $Y(\g_n)\hookrightarrow Y(\gl_{t_n})$. Being an embedding is a first-order property, so it follows from \L o\'s's theorem, that \eqref{eqn: embedding Y(g) into Y(gl_t)} defines an embedding.
\end{proof}
This map also embeds the special twisted Yangian $SY(\g)$ into $Y(\xsl_t)$.

\subsection{Finite-length representations of \texorpdfstring{$Y(\g)$}{Y(g)}}
Here again we set $t$ to be transcendental and $\g=\o_t$ or $\xsp_{t}$. Recall $\g_n = \o_{2t_n}$ or $\xsp_{2t_n}$ and $G_n=O_{2t_n}$ or $Sp_{2t_n}$ correspondingly. Also, when $t$ is transcendental, set $p_n=0$.

First we define the category of $Y(\g)$-modules.

\begin{defn}\label{defn: Y(g)-modules}
    Denote by $\Rep_0(Y(\g))$ the category with objects being objects $M\in \Rep(G)$ together with an element $\mu_M\in \Hom(Y(\g)\otimes M,M)$ such that
    \begin{enumerate}
        \item $M$ is a representation of $Y(\g)$, i.e. if $\mu$ is the product map of the Yangian, then $\mu_M\circ (1\otimes \mu_M) = \mu_M\circ (\mu\otimes 1)$ as elements of $\Hom(Y(\g)\otimes Y(\g)\otimes M,M)$.
        \item The map $\mu_M\circ (i\otimes 1)$ gives the standard structure of a $\g$ representation on $M$.
    \end{enumerate}
\end{defn}
The morphisms in $\Rep_0(Y(\g))$ are morphisms of $\Rep(G)$, which commute with the representation structure. Note that we consider only honest objects of $\Rep(G)$ and not ind-objects.

Similarly to \cref{thm: ultraproduct construction,thm: ultraproduct construction 2}, we can give an ultraproduct construction of $\Rep_0(Y(\g))$.

\begin{lem}\label{lem: irreducible objects of Y(g)}
    The category $\Rep_0(Y(\g))$ is defined as the full subcategory of \[\prod_{\mcF} \bfRep_0(Y(\g_{n}),p_n)\] consisting of those objects whose image in $\prod_{\mcF} \bfRep_{p_n}(G_{n})$ lies in $\Rep(G)$. Moreover, the irreducible representations of $Y(\g)$ correspond to collections of representations of $Y(\g_{n})$ such that almost all are irreducible. In other words, $\Irr(\Rep_0(Y(\g)))$ is the full subcategory of \[\prod_{\mcF} \Irr(\bfRep_0(Y(\g_{n}),p_n)),\] whose image is contained in $\Rep(G)$.
\end{lem}
\begin{proof}
    The proof is exactly the same as \cite[Lemma 4.2.2]{kalinov_yangians_20}.
\end{proof}
The same statement of course holds for representations of the special twisted Yangian.

It follows that for $M\in \Rep_0(Y(\g))$, we can write $M = \prod_{\mcF}M_n$, where $M_n$ is a representation of $Y(\g_{t_n})$. If $M$ is irreducible, by \L o\'s's theorem, almost all of the $M_n$ are irreducible. The analogous statement holds for the special twisted Yangian.

Consider the case where $t$ is transcendental and $\g = \o_t$. Suppose that $M \in \Rep_0(SY(\g))$. Then, almost all of the $Y(\g_n)$-modules $M_n$ corresponding to $M$ must be irreducible. Consequently, for almost all $n$, $M_n$ corresponds to a choice of the parameter $\gamma$ belonging to one of the three sets in~\eqref{eqn: subsets}. We refer to them as types $A$, $B$, and $C$, respectively.

Recall the $^\sharp$-automorphism of $Y(\o_{2n})$ defined in~\eqref{eqn: sharp automorphism}. We also denote by $^\sharp$ its ultraproduct, which, by \L oś's theorem, defines an automorphism of $Y(\o_t)$. If $M$ is of type $C$, let $M^\sharp$ be the $Y(\o_t)$-module obtained by composing the $^\sharp$-automorphism with the action on $M$. Then $M^\sharp$ is of type $B$. Finally, if $t$ is algebraic, we regard all irreducible representations of $Y(\o_t)$ as being of type $A$.

\begin{lem}\label{lem: positive char lemma}
    Fix $c\in \ZZ$. For algebraic $t$ and for almost all $n$, we have $p_n-2t_n >c$.
\end{lem}
\begin{proof}
    Suppose to the contrary that $p_n-2t_n\le c$ for almost all $n$. Since there are finitely many possibilities, $p_n-2t_n=d$ for some $d$. This means \[q(2t_n)=q(-d)\pmod{p_n},\] so $q(-d)$ has an infinite number of prime divisors, implying $q(-d)=0$. However, $q(x)$, being the minimal polynomial of $t$, cannot have integral roots, a contradiction.
\end{proof}

\begin{prop}\label{prop: subquotient of eval modules}
    Suppose $M\in \Rep_0(SY(\g))$ is irreducible. Then, $M\sqsubset L(\lambda(u))$ unless $\g = \o_t$ and $M$ is type $C$, in which case $M^\sharp\sqsubset L(\lambda(u))$. Here $L(\lambda(u))$ is a irreducible $Y(\xsl_t)$-module,
    considered as a $SY(\g)$-module via the embedding \eqref{eqn: embedding Y(g) into Y(gl_t)} and $\lambda(u)$ is a highest weight for $Y(\gl_t)$, which is a pair of sequences $(\lambda_i^{\bullet}(u),\lambda_i^{\circ}(u))$.
\end{prop}
\begin{proof}
    Since every irreducible module of the special twisted Yangian arises as the restriction of an irreducible $Y(\g)$-module (a consequence of Łoś's theorem, as the claim holds in integral rank), consider the irreducible $Y(\g)$-module corresponding to $M$, which we also denote by $M$. Let $\{M_n\}$ be the sequence of $Y(\g_{t_n})$-modules associated with $M$. Since $M$ is irreducible, \L o\'s's theorem implies that $M_n$ is irreducible for almost all $n$.

As an object of $\Rep(G)$, $M$ decomposes as a finite direct sum of highest-weight irreducible representations, i.e.,
\[
M = \bigoplus_{j=1}^K V(\mu_j),
\]
for some negative partitions $\mu_j$. Then, for almost all $n$, we have
\begin{align}\label{eqn: decomposition as a Y(g)-module}
    M_n = \bigoplus_{j=1}^K V([\mu_j]_{t_n},p_n).
\end{align}

\noindent\textbf{Symplectic Case:} Suppose $M_n$ is irreducible as a $Y(\g_{t_n})=Y(\xsp_{2t_n})$-module. Since $\max_i -([\mu_i]_{t_n})_{t_n}$ stabilizes for large $n$, it follows from \cref{lem: positive char lemma} that $M_n$ satisfies the conditions of \cref{prop: positive char highest weight} for almost all $n$.

By \cref{thm: Y(sp2n) classification 1} and the proof of \cref{thm: Y(sp2n) classification 2} (cf. \cite[the second part of Theorem 4.3.8]{molev_yangians_book_07}), $M_n$ is obtained by composing an automorphism $\mcS(u) \mapsto g_n(u)\mcS(u)$, where \[g_n(u) = 1 + g_1^{(n)} u^{-2} + g_2^{(n)} u^{-4} + \cdots\] is an even series, with the action on a subquotient $M_n'$ of the finite-dimensional irreducible $Y(\gl_{2t_n})$-module $L(\lambda^{(n)}(u))$, regarded as a $Y(\xsp_{2t_n})$-module via the embedding \eqref{eqn: embed Y(gn) into Y(glN)}.

Since $SY(\xsp_{2t_n})$ is invariant under the automorphism $\mcS(u) \mapsto g_n(u)\mcS(u)$, we may replace $M_n$ with $M_n'$, obtaining isomorphic $SY(\xsp_t)$-modules via ultraproducts: \[M = \prod_{\mcF} M_n \cong \prod_{\mcF} M_n' = M'.\] Furthermore, the embedding sends $SY(\xsp_{2t_n})$ into $Y(\xsl_{2t_n})$, allowing us to view $L(\lambda^{(n)}(u))$ as a $Y(\xsl_{2t_n})$-module. 

By construction, $\lambda^{(n)}(u)$ has a sequence of Drinfeld polynomials. Moreover, if another weight $\mu^{(n)}(u)$ has the same Drinfeld polynomials, then $\lambda^{(n)}(u)$ and $\mu^{(n)}(u)$ differ by multiplication by some even power series, so $L(\lambda^{(n)}(u)) \cong L(\mu^{(n)}(u))$ as $Y(\xsl_{2t_n})$-modules.

By \cite[Theorem 3.2.6]{kalinov_yangians_20}, $L(\mu^{(n)}(u))$ decomposes as a tensor product of evaluation modules:
\[
L(\mu^{(n)}(u)) \cong L(\lambda_1^{(n)},p_n) \otimes \cdots \otimes L(\lambda_{k_n}^{(n)},p_n).
\]
It follows that:
\[
M_n \sqsubset L(\lambda^{(n)}(u)) \sqsubset L(\lambda_1^{(n)},p_n) \otimes \cdots \otimes L(\lambda_{k_n}^{(n)},p_n),
\]
where the second subquotient is to be interpreted as $Y(\xsl_{2t_n})$-modules, and the $\lambda_i^{(n)}$ are dominant integral $\gl_{2t_n}$-weights.

The highest weight of $M$ with respect to the $\xsp_{2t_n}$-action is $\widetilde{\lambda}_1^{(n)} + \cdots + \widetilde{\lambda}_{k_n}^{(n)}$, where \[(\widetilde{\lambda}_j^{(n)})_i = (\lambda_j^{(n)})_i - (\lambda_j^{(n)})_{ - i}.\] We can restrict to $n$ with $2t_n > \max(l(\mu_j))$, since all cofinite subsets are in $\mcF$.

Since $M_n$ has a unique highest weight vector, this highest weight must be $[\mu_j]_{2t_n}$ for some $j$. Moreover, for all larger $n$, this remains the highest weight, as it depends only on the support of $[\mu_j]_{2t_n}$. Let $\mu$ be the corresponding negative partition, so that:
\[
[\mu]_{t_n} = \widetilde{\lambda}_1^{(n)} + \cdots + \widetilde{\lambda}_{k_n}^{(n)}.
\]

Since $[\mu]_{t_n}, \widetilde{\lambda}_1^{(n)}, \ldots, \widetilde{\lambda}_{k_n}^{(n)}$ are all negative partitions, and the negative partition associated to $[\mu]_{t_n}$ is constant for large $n$, there are finitely many choices for the $\widetilde{\lambda}_j^{(n)}$, so one such choice must occur for almost all $n$. Denote these limiting negative partitions by $\widetilde{\lambda}_1, \ldots, \widetilde{\lambda}_k$.

For each $\lambda_i^{(n)}$, we have the decomposition \[\lambda_i^{(n)} = \chi_i^{(n)} + d_i^{(n)},\] where $\chi_i^{(n)}$ is a $\gl_{2s_n}$-weight with $(\chi_i^{(n)})_1=0$ for large enough $n$, and $d_i^{(n)} \in \overline{\FF}_{p_n}$ is constant along all components. Then, $\widetilde{\lambda}_i = \widetilde{\chi}_i^{(n)}$.

For large enough $n$, we must have $(\widetilde{\chi}_i^{(n)})_1=0$ so $(\chi_i^{(n)})_1=0$ implies $(\chi_i^{(n)})_{-1}=0$. Similarly, we get $(\chi_i^{(n)})_k=0$ for $k=-t_n+l(\mu),\ldots, t_n-l(\mu)$; so the entries of $\chi_i^{(n)}$ have at most $l(\mu)$ positive and $l(\mu)$ negative parts. Thus there are finitely many possibilities for $\chi_i^{(n)}$, so it is eventually fixed for almost all $n$. Hence there exists a bipartition $\eta_i$ such that $[\eta_i]_{2t_n} = \chi_i^{(n)}$ for almost all $n$. Thus,
\[
M_n \sqsubset L([\eta_1]_{2t_n} + c_1^{(n)},p_n) \otimes \cdots \otimes L([\eta_k]_{2t_n} + c_k^{(n)},p_n)
\]
for some constants $c_i^{(n)}$. Setting $c_i = \prod_{\mcF} c_i^{(n)}$, and using \L o\'s's theorem, we get:
\[
M \sqsubset L(\eta_1 + c_1) \otimes \cdots \otimes L(\eta_k + c_k).
\]

Additionally, we have the following:
\[
L(\lambda^{(n)}(u)) \sqsubset L([\eta_1]_{2t_n} + c_1^{(n)},p_n) \otimes \cdots \otimes L([\eta_k]_{2t_n} + c_k^{(n)},p_n),
\]
so the ultraproduct $L=\prod_{\mcF} L(\lambda^{(n)}(u))$ lies in $\Rep(GL_t)$. By \cite[Definition 4.2.9]{kalinov_yangians_20}, $L$ is isomorphic to $L(\lambda(u))$ for some $Y(\gl_t)$-highest weight $\lambda(u)$, where $\lambda(u)$ is a pair of sequences $(\lambda_i^{\bullet}(u), \lambda_i^{\circ}(u))$. Then, by \L o\'s's theorem:
\[
M \sqsubset L(\lambda(u)).
\]

\noindent\textbf{Orthogonal Case:} The argument is analogous. For algebraic $t$ (i.e., in positive characteristic), the proof proceeds identically. In characteristic zero, it applies for types $A$ and $B$ modules since, if $M_n$ is type $A$ or $B$, the second part of the proof of \cite[Theorem 4.4.14]{molev_yangians_book_07} implies that $M_n$ is a subquotient of $L(\lambda^{(n)}(u))$, which in turn is a subquotient of $L(\lambda_1^{(n)}) \otimes \cdots \otimes L(\lambda_{k_n}^{(n)})$. If $M_n$ is type $C$, then $M_n^\sharp$ is of type $B$, and $M_n^\sharp \sqsubset L(\lambda(u))$.
\end{proof}

Also, note that the $Y(\xsl_t)$-module $L(\lambda(u))$ completely determines $M$, by the same argument as in  \cite[Corollary 4.2.8]{kalinov_yangians_20}. Let $M^\natural = M$, $M_n^\natural = M_n$ unless $M$ is type $C$ in which case set $M^\natural = M^\sharp$ and $M_n^\natural = M_n^\sharp$.
Then we have
\[M^\natural \sqsubset L(\lambda(u))\sqsubset L(\eta_1+c_1)\otimes \cdots \otimes L(\eta_k+c_k),\]
where $L(\lambda(u))$ is the finite-dimensional irreducible highest-weight representation of $Y(\gl_t)$, and $\lambda(u)$ is a pair of sequences $(\lambda_i^{\bullet}(u), \lambda_i^{\circ}(u))$ in $\CC\dbrack{u^{-1}}$ given by
\[\lambda_i^{\bullet}(u) =\prod_{j=1}^k (1+[(\eta_j^{\bullet})_i+c_j]u^{-1}) \quad \text{and}\quad \lambda_i^{\circ}(u) =\prod_{j=1}^k (1+[(\eta_j^{\circ})_i-c_j]u^{-1})\]
along with an element \[\lambda^m =\prod_k(1+c_ku^{-1}).\] These are defined up to simultaneous multiplication by any $f(u)\in \CC\dbrack{u^{-1}}$. We define the highest weight of $M^\natural$, as a $Y(\g)$-module, as a sequence $\CC\dbrack{u^{-1}}$ equal to
\[\mu_i(u) = \lambda_i^{\bullet}(-u) \lambda_i^{\circ}(-u)\]
along with an element \[\mu^m(u) = \lambda^m(-u)\lambda^m(u).\] Set $l(\mu) = l(\lambda)$. Note that multiplying $\lambda$ by $f(u)\in \CC\dbrack{u^{-1}}$ multiplies $\mu$ by $f(u)f(-u)\in \CC\dbrack{u^{-2}}$. Thus $\mu$ is defined up to simultaneous multiplication by $g(u)\in\CC\dbrack{u^{-2}}$. Since $\mu^m(u)$ is an even sequence, we may assume that $\mu^m(u) = 1$. So with this assumption, $\mu(u)$ is uniquely determined.

We can relate the highest weight of $M^\natural$ to the highest weights of $M_n^\natural$ in the following way. For $n>l(\mu)$, define 
\[([\mu]_{t_n})_{t_n+1-i}= \begin{cases}
    \mu_i(u) & \text{for } i\le l(\mu), \\
    \mu^m(u) & \text{otherwise},
\end{cases}\]
where the coefficients are also reduced to $\overline{\FF}_{p_n}$, i.e. $c_j$ is replaced with $c_j^{(n)}$, etc.

By \L o\'s's theorem, for almost all $n$, we have
\[M_n' \sqsubset L([\lambda]_{2t_n}(u)) \sqsubset L([\eta_1]_{2t_n}+c_1^{(n)},p_n)\otimes \cdots \otimes L([\eta_k]_{2t_n}+c_k^{(n)},p_n).\]
From \cite[Corollary 4.2.12]{molev_yangians_book_07}, the $t_n+1-i$-th component of the highest weight of $M_n^\natural$ is simply \[([\lambda]_{2t_n})_{-t_n-1+i}(-u)([\lambda]_{2t_n})_{t_n+1-i}(u) ,\] which is easily checked to equal $([\mu]_{t_n})_{t_n+1-i}$.

We can now prove the main classification theorems.
\begin{thm}\label{thm: SY(spt)}
    For every irreducible $M\in \Rep_0(SY(\xsp_t))$, there exists a sequence of monic polynomials, denoted by $P(u)=P_1(u),P_2(u),\ldots$ such that
    \begin{itemize}
        \item the corresponding highest weight satisfies
    \begin{align}\label{eqn: sp_t drinfeld polynomials}
    \frac{\mu_{i+1}(u)}{\mu_{i}(u)}=\frac{P_i(u+1)}{P_i(u)}.
    \end{align}
    \item for sufficiently large $i$, the $\mu_i(u)$ stabilizes and equals $\mu^m(u)=1$, and the $P_i(u)$ stabilize and equal $1$.
    \end{itemize}
    Moreover, for any such $P(u)$, there is $M$ with highest weight satisfying \eqref{eqn: sp_t drinfeld polynomials}. This also gives a one-to-one corresponence between irreducible modules and sequences of monic polynomials $P(u)$.
\end{thm}
The polynomials in the sequence $P(u)$ are referred to as the \textit{Drinfeld polynomials} corresponding to $M$.
\begin{proof}
    For the first part, consider the corresponding sequence of modules $M_n \in \Rep_0(SY(\xsp_t))$, almost all of which are irreducible with highest weight equal to $[\mu]_{t_n}(u)$, where $\mu(u)$ is the highest weight of $M$. Let the Drinfeld polynomials corresponding to $[\mu]_{t_n}(u)$ be \[Q^{(n)}(u), P_{t_n-1}(u),\ldots, P_1^{(n)}(u).\] Then, we have
\[
\frac{([\mu]_{t_n})_{t_n+1 - (i+1)}(u)}{([\mu]_{t_n})_{t_n+1-(i)}(u)} = \frac{P_{i}^{(n)}(u+1)}{P_{i}^{(n)}(u)}.
\]
Since \[\deg Q^{(n)}(u) + 2\sum_{i=1}^{t_n-1} \deg P_{i}^{(n)}(u) + 2t_n<p,\] the same argument as in the proof of  \cite[Theorem 4.2.11]{kalinov_yangians_20} applies to show that there exist polynomials $P_1(u), P_2(u), \ldots$ such that \eqref{eqn: sp_t drinfeld polynomials} holds. Moreover, as $\mu_i(u)$ stabilizes and equals $\mu^m(u)=1$ for sufficiently large $i$, it follows that $P_i(u) = 1$ for all sufficiently large $i$.

For the second statement, we may simply use an analogous argument to the second part of the proof of \cite[Theorem 4.2.11]{kalinov_yangians_20}.

Finally, note that the highest weight of $M$ is unique subject to the condition $\mu^m(u)=1$, so the Drinfeld polynomials are uniquely determined.
\end{proof}
\begin{thm}
    For every irreducible $M\in \Rep_0(SY(\o_t))$, $M$ is type $A$ and there exists a sequence of monic polynomials, denoted by $P(u)=P_1(u),P_2(u),\ldots$, such that
    \begin{itemize}
        \item the corresponding highest weight satisfies
    \begin{align}\label{eqn: o_t drinfeld polynomials}
    \frac{\mu_{i+1}(u)}{\mu_{i}(u)}=\frac{P_i(u+1)}{P_i(u)}.
    \end{align}
    \item for sufficiently large $i$, the $\mu_i(u)$ stabilizes and equals $\mu^m(u)=1$, and the $P_i(u)$ stabilize and equal $1$.
    \end{itemize}
    Moreover, for any such $P(u)$, there is $M$ with highest weight satisfying \eqref{eqn: o_t drinfeld polynomials}. This also gives a one-to-one correspondence between irreducible modules and sequences of monic polynomials $P(u)$.
\end{thm}
\begin{proof}
    If $t$ is algebraic, the proof is identical to that for the symplectic case. Now suppose $t$ is transcendental. We know that $\mu^m(u)=1$. Thus, for sufficiently large $n$, the highest weight $([\mu]_{n})(u)$ of $M_n^\natural$ satisfies $([\mu]_{n})_1(u) = 1$. It also must satisfy \eqref{eqn: Y(o2) classification positive char}. The $\gamma$ corresponding to it is in one of the first two subsets in \eqref{eqn: subsets}, since $M_n^\natural$ is type $A$ or $B$. However, for $\gamma$ in the second subset, \eqref{eqn: Y(o2) classification positive char} impossible for $\mu(u)=1$, i.e. $\mu'(u) = 1+1/2u^{-1}$. So $\gamma = -1/2$, and $M_n^\natural$ is type $A$. Thus $M_n^\natural$ is type $A$ for almost all $n$, so $M$ is type $A$. The rest of the proof is exactly as in the symplectic case. 
\end{proof}

Finally, let us discuss the classification of $Y(\g)$-modules. We know that
\[Y(\g) = SY(\g) \otimes Z_{HC}(Y(\g)).\]
Hence a representation of $Y(\g)$ is given by a tensor product of a representation of $SY(\g)$ with a representation of $\CC\dbrack{z_2,z_4,\ldots}$. This is simply a sequence of complex numbers, represented as an even series $g(u)\in\CC\dbrack{u^{-2}}$. Two such representations corresponding to the same irreducible representations of $SY(\g)$ differ by multiplication by a one-dimensional representation of $Y(\g)$. Thus, a specific choice of $g(u) \in 1+u^{-2}\CC\dbrack{u^{-2}}$, determines a unique representation of $Y(\g)$. Thus we have the following.
\begin{thm}\label{thm: main result}
    Irreducible objects of $\Rep_0(Y(\g))$ are in one-to-one correspondence with tuples $(P(u),g(u))$, where $P(u)=P_1(u),P_2(u),\ldots$ is a sequence of Drinfeld polynomials and $g(u)\in 1+u^{-2}\CC\dbrack{u^{-2}}$. 
\end{thm}
It is interesting to note that the classification does not depend on whether $\g = \o_t$ or $\xsp_t$. One may expect this given that $\Rep(O_t) = \Rep(Sp_{-t})$ as tensor categories, and only a slight modification needs to be made to the braiding for them to be equivalent as symmetric tensor categories.

\bibliographystyle{amsalpha}
\bibliography{bib}
\newpage

\appendix
\section{The Sklyanin determinant in positive characteristic}\label{sec: appendix}
In this appendix, we define the Sklyanin determinant in positive characteristic and prove \cref{thm: tensor product decomposition}. Our approach follows Sections 2.5–2.9 of \cite{molev_yangians_book_07}. Since many proofs mirror those in the positive characteristic setting, we will only provide brief outlines where appropriate.  

Consider the $R$-matrix $R(u)$, which takes values in $(\End \CC^N)^{\otimes 2}$. We define the $R$-matrices $R_{ij}(u)$, acting on the $i$th and $j$th components of $(\CC^N)^{\otimes m}$. Similarly, for an $N \times N$ matrix  
\[
M \in (\End \CC^N) \otimes Y(\g_n)\dbrack{u^{-1}},
\]  
we define  
\[
M_i \in (\End \CC^N)^{\otimes m} \otimes Y(\g_n)\dbrack{u^{-1}}
\]  
as acting on the $i$th component of $(\CC^N)^{\otimes m}$.  We will refer to the following equation as the \textit{ternary relation}

\begin{align}\label{eqn: ternary relation}
    R(u-v)T_1(u)T_2(u)=T_2(u)T_1(u)R(u-v),
\end{align}
where $T(u)$ is from section \ref{subsec: twisted yangians in integer rank} and the subscript indicates which components it acts on. We now introduce the rational function $R(u_1, \dots, u_m)$ with values in $(\End \CC^N)^{\otimes m}$, defined as  
\begin{align}\label{eqn: R(u1...um)}
    R(u_1,\ldots,u_m) = (R_{m-1,m})(R_{m-2,m} R_{m-2,m-1}) \cdots (R_{1m} \cdots R_{12}),
\end{align}  
where $u_1, \dots, u_m$ are independent variables, and we abbreviate $R_{ij} = R_{ij}(u_i - u_j)$. Here, $R_{ij}(u)$ denotes the $R$-matrix acting on the $i$th and $j$th components of $(\End \CC^N)^{\otimes m}$. Using the Yang-Baxter equation (see \eqref{eqn: Yang-Baxter}) and the fact that we may commute $R_{ij}$ and $R_{kl}$ for pairwise distinct $i,j,k,l$, a simple induction shows
\begin{align}\label{eqn: R(u1...um) ii}
    R(u_1,\ldots,u_m) = (R_{12}\cdots R_{1m})\cdots (R_{m-2,m-1}R_{m-2,m})(R_{m-1,m}).
\end{align}
Define
\[T_i=T_i(u_i), \quad \mcS_i=\mcS_i(u_i), \qquad 1\le i\le m\]
and
\[R_{ij}' = R_{ji}' =R_{ij}'(-u_i-u_j).\]
For an arbitrary permutation $(p_1,\ldots,p_m)$ of the indices $1,\ldots,m$, we abbreviate
\begin{align}
    \left<\mcS_{p_1},\ldots, \mcS_{p_m}\right> = \mcS_{p_1}(R_{p_1p_2}'\cdots R_{p_1p_m}') \mcS_{p_2}(R_{p_2p_3}'\cdots R_{p_2p_m}')\cdots \mcS_{p_m}.
\end{align}
\begin{prop}
    We have the identity
    \begin{align*}
        R(u_1,\ldots, u_m)T_1\cdots T_m = T_m\cdots T_1R(u_1,\ldots,u_m).
    \end{align*}
\end{prop}
\begin{proof}
    First, by repeated application of the ternary relation \eqref{eqn: ternary relation} and the fact that $R_{ij}$ and $T_k$ are permutable for $i,j,k$ distinct, we get
    \[(R_{1m}\cdots R_{12}) T_1(T_2\cdots T_m) = (T_2\cdots T_m)T_1(R_{1m}\cdots R_{12}).\]
    Since
    \[R(u_1,\ldots,u_m) = R(u_2,\ldots,u_m)(R_{1m}\cdots R_{12}),\]
    by induction on $m$ we are done.
\end{proof}
\begin{prop}
    We have the identity
    \begin{align}
        R(u_1,\ldots, u_m) \left<\mcS_1,\ldots,\mcS_m\right> = \left<\mcS_m,\ldots,\mcS_1\right>R(u_1,\ldots u_m).
    \end{align}
\end{prop}
\begin{proof}
    Let $i,j,k$ be distinct indices in $\{1,\ldots,m\}$. Note that we have
    \begin{align}\label{eqn: quaternary relations}
        R_{ij}\mcS_i R_{ij}' \mcS_{j} &= \mcS_{j}R_{ji}'\mcS_iR_{ij} \\
        \label{eqn: Yang-Baxter final transposed}
        R_{ij}R_{ik}'R_{jk}' &= R_{jk}'R_{ik}'R_{ij}.
    \end{align}
    Indeed, \eqref{eqn: quaternary relations} is simply the quaternary relations, and for \eqref{eqn: Yang-Baxter final transposed}, we start with the Yang-Baxter equation
    \begin{align}\label{eqn: Yang-Baxter}
    R_{ij}R_{ik}R_{jk}=R_{jk}R_{ik}R_{ij}.
    \end{align}
    Then apply the transposition with respect to the matrix $G$ on the $k$'th component to get 
    \begin{align}\label{eqn: Yang-Baxter transposed}
    R_{ij}R_{jk}'R_{ik}' = R_{ik}'R_{jk}'R_{ij}.
    \end{align}
    Now we have $P_{ij}R_{jk}^tP_{ij}=R_{ik}^t$, where $P$ is the permutation operator and \[R_{jk}^t = R_{jk}^t(-u_j-u_k).\] Then we get that
    \[(G_{k}^{-1}P_{ij}G_{k})R_{jk}'(G_{k}^{-1}P_{ij}G_{k}) = (G_{k}^{-1}P_{ij}G_{k})(G_{k}^{-1}R_{jk}^tG_{k})(G_{k}^{-1}P_{ij}G_{k}) =  G_{k}^{-1}R_{ik}^t G_{k}\]
    which is simply $R_{ik}'$. So we conjugate both sides of \eqref{eqn: Yang-Baxter transposed} by $(G_{k}^{-1}P_{ij}G_{k})$, and since conjugating $R_{ij}$  gives us $R_{ij}$ again, since $R_{ij}$ and $G_k$ commute, and $P_{ij}R_{ij}P_{ij}=R_{ij}$, we get \eqref{eqn: Yang-Baxter final transposed}. From here, we can proceed in exactly the same way as the proof of Proposition 2.5.1 in \cite{molev_yangians_book_07}, replacing ``$A^t$'' transpose symbols with the ``$A'$'' transpose with respect to $G$.
\end{proof}

Consider the symmetric group $\mfS_m$ acting on $\{1,\ldots,m\}$. The anti-symmetrizer in the group algebra of $\mfS_m$ is given by
\[\sum_{\sigma\in \mfS_m}\sgn(\sigma)\cdot \sigma \in \CC[\mfS_m].\]
Let $A_m$ be the image of this anti-symmetrizer under the natural action of $\mfS_m$ on $(\CC^N)^{\otimes m}$. We use $e_1,\ldots, e_N$ to be the standard basis vectors of $\CC^N$, we get
\begin{align}\label{eqn: anti-symmetrizer}
    A_m(e_{i_1}\otimes \cdots\otimes e_{i_m}) = \sum_{\sigma\in \mfS_m}\sgn(\sigma)\cdot e_{i_{\sigma(1)}}\otimes \cdots \otimes e_{i_{\sigma(m)}}.
\end{align}
Note that we have $A_m^2 = m!A_m$. For $m\le N$ the right hand side is nonzero because we assume $N<p$, see \cref{prop: positive char highest weight}.
\begin{prop}\label{prop: R(u1...um)=Am}
    For $m\le N$, if $u_i-u_{i+1}=1$ for all $i=1,\ldots,m-1$, then
    \[R(u_1,\ldots,u_m) = A_m.\]
\end{prop}
\begin{proof}
    We can use the same proof as Proposition 1.6.2 in \cite{molev_yangians_book_07}. There are no division by zero problems because all denominators are $\le m\le N<p$.
\end{proof}
Now set $m = N$ and  
\[
u_i = u - i + 1, \quad i = 1, \dots, N.
\]
By the previous propositions,  
\begin{align}\label{eqn: qdet definition}
    A_N T_1\cdots T_m = T_m\cdots T_1 A_N
\end{align}
and  
\begin{align}\label{eqn: sdet definition}
    A_N\left< S_1,\ldots, S_N\right> = \left<S_N,\ldots,S_1\right>A_N.
\end{align}

The image of $A_N$ on $(\CC^N)^{\otimes N}$ is one-dimensional. Indeed, for any basis element $e_{i_1} \otimes \dots \otimes e_{i_N}$, if $i_k = i_{\ell}$ for some $k \neq \ell$, then by \eqref{eqn: anti-symmetrizer},  
\[
A_N(e_{i_1} \otimes \dots \otimes e_{i_N}) = 0.
\]
Thus, $A_N$ is nonzero only on pure tensors indexed by permutations of $(1, \dots, N)$, where it evaluates to  
\[
\sum_{\sigma\in \mfS_N} \sgn(\sigma) \cdot e_{\sigma(1)} \otimes \dots \otimes e_{\sigma(N)}.
\]
It follows that \eqref{eqn: qdet definition} and \eqref{eqn: sdet definition} reduce to $A_N$ multiplied by a scalar series with coefficients in $Y(\gl_N)$ or $Y(\g_n)$, respectively.

 
\begin{defn}
    The \emph{quantum determinant} (resp. \emph{Skylanin determinant}) of the matrix $T(u)$ (resp. $\mcS(u)$) is the formal series with coefficients in $Y(\gl_N)$ (resp. $Y(\g_n)$) given by
    \[\qdet T(u) = d_0+d_1u^{-1}+d_2u^{-2}+\cdots,\]
    respectively
    \[\sdet \mcS(u) = c_0+c_1u^{-1}+c_2u^{-2}+\cdots,\]
    such that the element \eqref{eqn: qdet definition} (resp \eqref{eqn: sdet definition}) equals $A_N\qdet T(u)$ (resp. $A_N\sdet \mcS(u)$).
\end{defn}
\begin{prop}\label{prop: qdet positive char}
    We have for any permutation $\pi\in\mfS_N$
    \begin{align*}
        \qdet T(u) &= \sgn(\pi)\sum_{\sigma\in\mfS_N}\sgn(\sigma)\cdot t_{\sigma(1),\pi(1)}(u)\otimes \cdots \otimes t_{\sigma(N),\pi(N)}(u-N+1)\\
        &=  \sgn(\pi)\sum_{\sigma\in\mfS_N}\sgn(\sigma)\cdot t_{\pi(1),\sigma(1)}(u-N+1)\otimes \cdots \otimes t_{\pi(N),\sigma(N)}(u).
    \end{align*}
\end{prop}
\begin{proof}
    The characteristic zero proof works just as well in positive characteristic, see the proof of Proposition 1.6.6 in \cite{molev_yangians_book_07}.
\end{proof}

Due to this, we see that this definition of the quantum determinant matches the one in \cite{kalinov_yangians_20}, so we know that its coefficients are central and algebraically independent.
\begin{thm}\label{thm: sdet and qdet}
    Under the embedding \eqref{eqn: embed Y(gn) into Y(glN)},
    \[\sdet \mcS(u) = \alpha_N(u)\qdet T(u)\qdet T(-u+N-1),\]
    where $\alpha_N(u)$ is given by \eqref{eqn: alphaq(u)}.
\end{thm}
\begin{proof}
    Consider the mappings  
\begin{align*}
    T(u) &\mapsto G^{-1}T(u)G,\\
    T(u) &\mapsto T^t(u),\\
    T(u) &\mapsto T(-u).
\end{align*}
The first is an automorphism, while the second and third are anti-automorphisms by Propositions 1.3.1 and 1.3.3 in \cite{molev_yangians_book_07} ((1.22), (1.26), (1.25)). The proof follows from the characteristic-zero case. Their composition yields  
\begin{align}
    T(u) \mapsto G^{-1}T^t(-u)G = T'(-u),
\end{align}
an automorphism denoted $\tau_N$.  

From \eqref{eqn: embed Y(gn) into Y(glN)}, we obtain $S_i = T_i T'_i$, where $T_i = T_i(u-i+1)$ and $T_i^\sigma = T_i'(-u+i-1)$. Thus, the left-hand side of \eqref{eqn: sdet definition} rewrites as  
\begin{align}\label{eqn: LHS 1}
    A_N T_1T_1^\sigma R_{12}'\cdots R_{1N}' T_2T_2^{\sigma}R_{23}'\cdots R_{2N}' \cdots T_{N-1}T_{N-1}^\sigma R_{N-1,N}'T_NT_N^\sigma,
\end{align}
where $R_{ij}' = R_{ij}(-2u+i+j-2)$. Applying the ternary relation \eqref{eqn: ternary relation} and partial transpose with respect to $G$ on the first component,  
\[
T_1'(u)R'(u-v)T_2(v) = T_2(v)R'(u-v)T_1'(u).
\]
Thus,  
\begin{align}\label{eqn: transposed ternary}
    T_i^\sigma R_{ij}' T_j = T_j R_{ji}'T_i^\sigma.
\end{align}
For $i\neq j,k$, both $T_i$ and $T_i^\sigma$ commute with $R_{jk}'$, allowing us to rewrite \eqref{eqn: LHS 1} as  
\begin{align}
    A_N T_1 (T_1^\sigma R_{12}'T_2)R_{13}'\cdots R_{1N}'(T_2^{\sigma}R_{23}'T_3)\cdots (T_{N-1}^\sigma R_{N-1,N}' T_N)T_N^\sigma.
\end{align}
Applying \eqref{eqn: transposed ternary} repeatedly, we obtain  
\begin{align*}
    A_N T_1\cdots T_N R_{12}'\cdots R_{1N}' R_{23}'\cdots R_{2N}'\cdots R_{N-1,N}' T_1^\sigma\cdots T_N^\sigma.
\end{align*}
This simplifies to  
\begin{align*}
    \qdet T(u) A_N \left<I_1,I_2,\ldots,I_N\right> T_1^{\sigma}\cdots T_N^{\sigma},
\end{align*}
where $I_i$ is the identity operator. Since the mapping $\mcS(u) \mapsto I$ defines a representation of $T(\g_n)$, we obtain from \eqref{eqn: sdet definition}  
\begin{align}\label{eqn: betaN(u)}
    A_N \left<I_1,\ldots, I_n\right> = \left<I_N,\ldots,I_1\right>A_N = A_N\beta_N(u),
\end{align}
for some scalar function $\beta_N(u)$. Applying $\tau_N$ to \eqref{eqn: qdet definition},  
\[
A_NT_1^\sigma\cdots T_N^\sigma = A_N\tau_N(\qdet T(u)).
\]
Since  
\[
\qdet T(u) = \sum_{\sigma\in\mfS_N}\sgn(\sigma)\cdot t_{\sigma(1),1}(u)\cdots t_{\sigma(N),N}(u-N+1),
\]
and $\tau_N$ maps $T(u) \mapsto T'(-u)$, we find  
\begin{align*}
    \tau_N(\qdet T(u)) = \sum_{\sigma\in \mfS_N}\sgn(\sigma) \cdot t_{N,N+1-\sigma(1)}(-u)\cdots t_{1,N+1-\sigma(N)}(-u+N-1).
\end{align*}
Introducing $\pi(i) = N+1-i$, this rewrites as  
\begin{align*}
    &\sum_{\sigma\in \mfS_N}\sgn(\sigma) \cdot t_{\pi(1),(\pi\circ\sigma)(1)}(-u)\cdots t_{\pi(N),(\pi\circ\sigma)(N)}(-u+N-1)\\
    =&\sgn(\pi)\sum_{\sigma'\in \mfS_N}\sgn(\sigma') \cdot t_{\pi(1),\sigma'(1)}(-u)\cdots t_{\pi(N),\sigma'(N)}(-u+N-1).
\end{align*}
Thus, $\tau_N(\qdet T(u)) = \qdet T(-u+N-1)$, leading to  
\[
A_N\sdet \mcS(u) = A_N \beta_N(u)\qdet T(u) \qdet T(-u+N-1).
\]
Applying this to $e_1\otimes \cdots \otimes e_N$,  
\begin{align*}
    &\sdet \mcS(u)\sum_{\sigma\in \mfS_N} \sgn(\sigma) e_{\sigma(1)}\otimes \cdots \otimes e_{\sigma(N)}\\
    =&\beta_N(u)\qdet T(u) \qdet T(-u+N-1)\sum_{\sigma\in \mfS_N} \sgn(\sigma) e_{\sigma(1)}\otimes \cdots \otimes e_{\sigma(N)}.
\end{align*}
It follows that  
\begin{align*}
    \sdet \mcS(u) =\beta_N(u)\qdet T(u) \qdet T(-u+N-1).
\end{align*}
To show $\beta_N(u) = \alpha_N(u)$, consider first the orthogonal case. With  
\[
R(u) = 1 - \frac{P}{u}, \quad P = \sum_{i,j}e_{ij}\otimes e_{ji},
\]
we obtain  
\[
R'(u) = 1 - \frac{P'}{u}, \quad P' = \sum_{i,j}e_{ij}\otimes e_{N+1-i,N+1-j}.
\]
By direct computation, $PP' = P'$, so  
\[
(1-P)R'(u) = (1-P).
\]
For $i<j$, this implies  
\[
A_N R_{ij}' = A_N.
\]
Thus, $\alpha_N(u) = 1$, agreeing with \eqref{eqn: alphaq(u)}. In the symplectic case, applying $R'_{jk}$ to the basis vector  
\[
v = e_1 \otimes e_2 \otimes \cdots \otimes e_n \otimes e_{n+1} \otimes \cdots \otimes e_{2n}
\]
yields  
\[
\frac{2u-2n+3}{2u-2n+1} A_{2n} v.
\]
Continuing this process leads to  
\[
\frac{2u+1}{2u-2n+1} A_{2n} v,
\]
matching \eqref{eqn: alphaq(u)} and proving $\beta_N(u) = \alpha_N(u)$.
\end{proof}

As a consequence of this theorem, the coefficients of $\sdet \mcS(u)$ are central, allowing us to define $Z_{HC}(\g_n)$ as the algebra generated by these coefficients.

Applying \cref{thm: sdet and qdet}, we immediately obtain  
\[
\alpha_N(-u+N-1) \cdot \sdet \mcS(u) = \alpha_N(u) \cdot \sdet \mcS(-u+N-1).
\]
From this, it follows that all the odd coefficients can be expressed in terms of the even ones. Consequently, $Z_{HC}(Y(\gl_n))$ is generated by the even coefficients. Finally, the fact these coefficients are algebraically independent follows in the same way as the characteristic zero case; see the proof of Theorem 2.8.2 in \cite{molev_yangians_book_07}.

Next, we construct a series $\widetilde{d}(u) \in 1 + u^{-1}Z_{HC}(Y(\gl_n))\dbrack{u^{-1}}$ satisfying  
\[
\widetilde{d}(u)\widetilde{d}(u-1)\cdots \widetilde{d}(u-N+1) = \qdet T(u) = 1 + d_1 u^{-1} + d_2 u^{-2} + \cdots.
\]
Writing  
\[
\widetilde{d}(u) = 1 + \widetilde{d}_1 u^{-1} + \widetilde{d}_2 u^{-2} + \cdots,
\]
expanding the power series reveals that the $u^{-k}$ coefficient of $\widetilde{d}(u-m)$ takes the form $\widetilde{d}_k$ plus terms depending only on $\widetilde{d}_i$ for $i < k$. Thus, the $u^{-k}$ coefficient of $\widetilde{d}(u)\widetilde{d}(u-1)\cdots \widetilde{d}(u-N+1)$ is $N\widetilde{d}_k$ plus terms involving only $\widetilde{d}_i$ for $i < k$. Since $N < p$, we can recursively determine the $d_i$ starting from $k=1$.  

By \cref{prop: qdet positive char}, automorphisms of the form \eqref{eqn: power series automorphism} map  
\[
\widetilde{d}(u) \mapsto f(u)\widetilde{d}(u).
\]
Finally, we define  
\[
\widetilde{t}_{ij}(u) = \widetilde{d}(u)^{-1} t_{ij}(u).
\]
It follows that $\widetilde{t}_{ij}(u)$ remains invariant under automorphisms of the form \eqref{eqn: power series automorphism}, implying that $\widetilde{t}_{ij}(u) \in Y(\xsl_N)$.

Now we can finish the proof of \cref{thm: tensor product decomposition}. 

    First, we establish the decomposition  
\begin{align}\label{eqn: not tensor product decomposition}
    Y(\g_n) = Z_{HC}(Y(\g_n)) SY(\g_n).
\end{align}
We regard $Y(\g_n)$ as a subalgebra of $Y(\gl_N)$ via the embedding \eqref{eqn: embed Y(gn) into Y(glN)}. Under this embedding, we have  
\[
\s_{ij}(u) = \sum_k \theta_{kj} t_{ik}(u) t_{N+1-j, N+1-k}(-u).
\]
Define the series  
\[
\widetilde{\s}_{ij}(u) = (\widetilde{d}(u) \widetilde{d}(-u))^{-1} \s_{ij}(u).
\]
Since $\widetilde{d}(u)$ has central coefficients, it follows that  
\[
\widetilde{\s}_{ij}(u) = \sum_k \theta_{kj} \widetilde{t}_{ik}(u) \widetilde{t}_{N+1-j, N+1-k}(-u).
\]
Now, we compute  
\begin{align*}
    &\alpha_N(u)^{-1} \sdet \mcS(u) = \qdet T(u) \qdet T(-u+N-1) \\
    &= (\widetilde{d}(u) \widetilde{d}(-u)) (\widetilde{d}(u-1) \widetilde{d}(-u+1)) \cdots (\widetilde{d}(u-N+1) \widetilde{d}(-u+N-1)).
\end{align*}
From this, it follows that all coefficients of $\widetilde{d}(u) \widetilde{d}(-u)$ can be expressed as polynomials in the coefficients of $\sdet \mcS(u)$, implying that $\widetilde{d}(u) \widetilde{d}(-u) \in Z_{HC}(\g_n)$.  

Similarly to how $\widetilde{t}_{ij}(u)$ belongs to $Y(\xsl_N)$, we deduce that $\widetilde{\s}_{ij}(u) \in SY(\g_n)$. Thus, the decomposition  
\[
\s_{ij}(u) = \widetilde{d}(u) \widetilde{d}(-u) \widetilde{\s}_{ij}(u)
\]
yields the desired result \eqref{eqn: not tensor product decomposition}.  

Finally, the tensor product decomposition follows from the relation  
\[
Y(\gl_N) = Z_{HC}(Y(\gl_N)) \otimes Y(\xsl_N),
\]
along with the inclusions $Z_{HC}(Y(\g_n)) \subset Z_{HC}(Y(\gl_N))$ (which follows from \cref{thm: sdet and qdet}) and $SY(\g_n) \subset Y(\xsl_N)$ (by the same argument as in characteristic zero).

This completes the proof of \eqref{thm: tensor product decomposition}.
\end{document}